\documentclass[10pt,a4paper]{article}
\usepackage[utf8]{inputenc}
\usepackage{amsmath}
\usepackage{amsfonts}
\usepackage{amssymb}
\usepackage{amsthm}
\usepackage{color}
\usepackage{enumerate}
\usepackage{marvosym}
 \newtheorem{thm}{Theorem}[section]
 \newtheorem*{thmA}{Theorem A}
 \newtheorem*{thmB}{Theorem B}
 \newtheorem*{thm*}{Theorem}
 \newtheorem{cor}[thm]{Corollary}
 \newtheorem{lem}[thm]{Lemma}
 \newtheorem{prop}[thm]{Proposition}
 \theoremstyle{definition}
 \newtheorem{defn}[thm]{Definition}
 \theoremstyle{remark}
 \newtheorem{rem}[thm]{Remark}

 \newtheorem*{exs}{Examples}
 \numberwithin{equation}{section}
 
\newcommand{\vertiii}[1]{{\left\vert\kern-0.25ex\left\vert\kern-0.25ex\left\vert #1
  \right\vert\kern-0.25ex\right\vert\kern-0.25ex\right\vert}}

\newcommand{\tr}{\operatorname{Tr}}

\newcommand{\Span}{\operatorname{span}}
\newcommand{\im}{\operatorname{Im}}

\newcommand{\closedspan}{\overline{\operatorname{span}}}

\newcommand{\buc}{\operatorname{BUC}}
\newcommand{\BUC}{\operatorname{BUC}}

\begin{document}
\title{Correspondence theory on $p$-Fock spaces with applications to Toeplitz algebras}
\author{Robert Fulsche}

\maketitle
\begin{abstract}
We prove several results concerning the theory of Toeplitz algebras over $p$-Fock spaces using a correspondence theory of translation invariant symbol and operator spaces. The most notable results are: The full Toeplitz algebra is the norm closure of all Toeplitz operators with bounded uniformly continuous symbols. This generalizes a result obtained by J. Xia \cite{Xia} in the case $p = 2$, which was proven by different methods. Further, we prove that every Toeplitz algebra which has a translation invariant $C^\ast$ subalgebra of the bounded uniformly continuous functions as its set of symbols is linearly generated by Toeplitz operators with the same space of symbols.


\end{abstract}

\section{Introduction}
In recent years, the study of Toeplitz algebras over Bergman or Segal-Bargmann-Fock spaces has experienced significant interest. Usually, one tries to understand the structure of an algebra, say a $C^\ast$ algebra, generated by Toeplitz operators with symbols from a subset of $L^\infty$ having some common property, which gives access to a deeper study. We mention \cite{Bauer_Hagger_Vasilevski2019, Esmeral_Vasilevski2016, Vasilevski_2008} and also refer to references therein. A joint approach in all those works was that they only deal with subalgebras of the full Toeplitz algebra. By the full Toeplitz algebra we mean the Banach algebra generated by all Toeplitz operators with $L^\infty$ symbols. In contrast to that, J. Xia in his paper \cite{Xia} proved a remarkable result about the full Toeplitz algebra. Let us fix some notation before going into details.

By $F_t^p$ we denote the $p$-Fock space, i.e. the space of holomorphic functions on $\mathbb C^n$ which are $p$-integrable with respect to a certain Gaussian measure with a parameter $t > 0$. We will denote the full Toeplitz algebra on $F_t^p$ by $\mathcal T^{p,t}$. By $\mathcal T^{p,t}(S) \subseteq \mathcal L(F_t^p)$ we mean the Banach algebra generated by Toeplitz operators with symbols in $S \subseteq L^\infty(\mathbb C^n)$, by $\mathcal T_{lin}^{p,t}(S) \subseteq \mathcal L(F_t^p)$ the closed linear space generated by such Toeplitz operators and by $\mathcal T_{\ast}^{2,t}(S) \subseteq \mathcal L(F_t^2)$ the $C^\ast$ algebra generated by these operators. The result by J. Xia is now the following:
\begin{thm*}[\cite{Xia}]
The following holds true:
\[ \mathcal T^{2,1} = \mathcal T_{lin}^{2,1}(L^\infty(\mathbb C^n)). \]
\end{thm*}
We want to stress that Xia also proved an analogous result for the Toeplitz algebra on the Bergman space over the unit ball in $\mathbb C^n$ in his paper. Further, with well-known methods it is possible to improve Xia's result to $\mathcal T^{2,1} = \mathcal T_{lin}^{2,1}(\BUC(\mathbb C^n))$. Here, $\BUC(\mathbb C^n)$ is the space of all bounded and uniformly continuous functions on $\mathbb C^n$.

While the assumption $t = 1$ was not crucial in Xia's proof, the restriction to $p = 2$ was important. Our first main theorem will be an improvement of that result:
\begin{thmA}
Let $1 < p < \infty$ and $t > 0$. Then, we have
\[ \mathcal T^{p,t} = \mathcal T_{lin}^{p,t}(\BUC(\mathbb C^n)). \]
\end{thmA}

We want to stress that our approach to this problem gives a more constructive result on how to approximate operators from $\mathcal T^{p,t}$ by Toeplitz operators than Xia's proof, which was based on an abstract $C^\ast$ algebraic argument. Our method of proof of Theorem A will lead us naturally to the study of translation-invariant algebras, both on the side of symbols and on the side of operator algebras. Here, we say that $\mathcal D_0 \subseteq \BUC(\mathbb C^n)$ is translation-invariant if $f(\cdot - z) \in \mathcal D_0$ for all $f \in \mathcal D_0, ~ z \in \mathbb C^n$. These investigations will lead us to our second main result. In the following, $U$ is the operator $Uf(z) = f(-z)$.
\begin{thmB}
Let $\mathcal D_0 \subseteq \BUC(\mathbb C^n)$ be closed, translation- and $U$-invariant. Then, the following are equivalent:
\begin{enumerate}[(i)]
\item $\mathcal D_0$ is a $C^\ast$ algebra with respect to the standard operations and $L^\infty$ norm;
\item $\mathcal T_{lin}^{2,t}(\mathcal D_0) = \mathcal T_{\ast}^{2,t}(\mathcal D_0)$ for all $t > 0$.
\end{enumerate}
If the above equivalent conditions are fulfilled, then we have $\mathcal T_{lin}^{p,t}(\mathcal D_0) = \mathcal T^{p,t}(\mathcal D_0)$ for all $1 < p < \infty, ~ t > 0$.

If further $\mathcal D_0$ is a closed, translation- and $U$-invariant $C^\ast$ subalgebra of $\BUC(\mathbb C^n)$ and $\mathcal I \subseteq \mathcal D_0$ is a closed and translation-invariant subset of $\mathcal D_0$, then the following are equivalent:
\begin{enumerate}[(i*)]
\item $\mathcal I$ is an ideal in $\mathcal D_0$;
\item $\mathcal T_{lin}^{2,t}(\mathcal I)$ is a one-sided ideal in $\mathcal T_{lin}^{2,t}(\mathcal D_0)$ for all $t > 0$;
\item $\mathcal T_{lin}^{2,t}(\mathcal I)$ is a two-sided ideal in $\mathcal T_{lin}^{2,t}(\mathcal D_0)$ for all $t > 0$.
\end{enumerate}
Under these assumptions, $\mathcal T_{lin}^{p,t}(\mathcal I) = \mathcal T^{p,t}(\mathcal I)$ is a closed and two-sided ideal in $\mathcal T_{lin}^{p,t}(\mathcal D_0)$ for all $1 < p < \infty, ~ t > 0$.
\end{thmB}
Let us add some words on the background of the method we will use. In the paper \cite{Werner1984}, R. Werner introduced his concept of \emph{Quantum Harmonic Analysis}, as he called it. By this, he in essence means a combination of two things: First of all, he introduced a concept of convolution between objects from $L^1(\mathbb R^{2n}) \times \mathcal N(L^2(\mathbb R^n))$ and objects from $L^\infty(\mathbb R^{2n}) \times \mathcal L(L^2(\mathbb R^n))$. Here, $\mathcal N(L^2(\mathbb R^n))$ denotes the trace class operators on the underlying Hilbert space. His concept of convolutions naturally extends the convolution between functions from $L^1(\mathbb R^{2n})$ and $L^\infty(\mathbb R^{2n})$. Using this convolution formalism (and a theory of \emph{regular operators}, which we will not need in this work), he then studies a natural correspondence between certain subspaces of $L^\infty(\mathbb R^{2n})$ and of $\mathcal L(L^2(\mathbb R^n))$. Based on ideas from that work, and extending them, we obtain the above results for Toeplitz algebras. Further ideas, which go into the proof of Theorem B, are results on quantization estimates \cite{Bauer_Coburn2015, Bauer_Coburn_Hagger} and limit operators \cite{Bauer_Isralowitz2012}, which have been studied out of independent interest and now fit nicely into our theory. Let us also mention that the correspondence theory, which we study in this work, gives rise to several other results, both old and new, with simple proofs. We name a view examples:
\begin{enumerate}
\item The characterization of compact operators from $\mathcal L(F_t^p)$, which was first proven in \cite{Bauer_Isralowitz2012}, is derived.
\item Generalized Schatten-von Neumann classes $\mathcal S^{p_0}(F_t^p)$ over $F_t^p$ are characterized as the closure of the space of Toeplitz operators with $L^{p_0}(\mathbb C^n)$ symbols.
\item We show that the algebras of Lagrangian-invariant Toeplitz operators are linearly generated. This was first obtained in \cite{Esmeral_Vasilevski2016}.
\end{enumerate}

The work is organized as follows: In part 2 we study Werner's Quantum Harmonic Analysis in the setting of $p$-Fock spaces. First, we introduce the appropriate Banach space version of Schatten-von Neumann ideals in Section 2.1. In Section 2.2 we discuss some basics on $p$-Fock spaces and certain natural operators on them. In Section 2.3, we introduce the notion of convolutions as motivated by Werner's work. In Section 2.4, we see that the well-known Toeplitz- and Berezin-maps can be nicely studied in Werner's convolution formalism. In Section 2.5 we fix several results which are part of Werner's Quantum Harmonic Analysis (as formulated in our scope of $p$-Fock spaces) but have not been noted directly in \cite{Werner1984}. They will come in handy later in our work. Section 2.6 concludes the development of Werner's theory in our setting by introducing his concept of Correspondence Theory. Part 3 is then dedicated to applying the theory to Toeplitz operators. Section 3.1 discusses several immediate consequences of the Correspondence Theory to the theory of Toeplitz algebras, e.g. the proof of Theorem A. Sections 3.2 and 3.3 are finally dedicated to the proof of Theorem B, which uses the already mentioned quantization estimates and limit operator techniques.

We want to mention the following: Section 2.3 discusses Werner's convolution formalism in some detail. Most of the theory can be immediately carried over from the original work, yet one must spend some care at certain points. Our presentation of the material in that section was closely inspired by the work \cite{Luef_Skrettingland2018}, which also discusses Werner's convolution formalism in the Schr\"odinger representation, but in closer detail. Section 2.6 follows the corresponding part of Werner's initial work closely and is included for the reader's convenience.

\section{Convolution formalism and correspondence theory over $p$-Fock spaces}

\subsection{Schatten-von Neumann ideals}
Let $X$ be a complex Banach space. By $\mathcal L(X)$ we denote the Banach algebra of all bounded linear operators on $X$. Recall that an operator $A \in \mathcal L(X)$ is \emph{nuclear} if there  are sequences $(x_j) \subset X, (y_j) \subset X'$ with $\sum_{j=1}^\infty \| y_j\|_{X'} \|x_j\|_{X} < \infty$ such that 
\begin{equation}\label{nuclear_operator}
A = \sum_{j=1}^\infty y_j \otimes x_j.
\end{equation}
For such an operator we define
\[ \| A\|_{\mathcal N} := \inf \sum_{j=1}^\infty \| y_j\|_{X'} \| x_j\|_{X}, \]
where the infimum is taken over all possible representations (\ref{nuclear_operator}). We denote by $\mathcal N(X)$ the set of all nuclear operators on $X$. Together with the norm $\| \cdot \|_{\mathcal N}$, this is well-known to be a Banach ideal in $\mathcal L(X)$. If the underlying Banach space $X$ has the approximation property, we can define the nuclear trace for $A \in \mathcal N(X)$ through
\[ \tr(A) = \sum_{j=1}^\infty y_j(x_j), \]
where the trace is independent of the choice of representation (\ref{nuclear_operator}), cf. \cite[Theorem V.1.2]{Gohberg_Goldberg_Krupnik2000}. From now on, we assume that $X$ has the approximation property and is also reflexive. Then, one can show that the duality relations
\[ (\mathcal K(X))' = \mathcal N(X), \quad (\mathcal N(X))' = \mathcal L(X) \]
hold true, where the duality is induced by the trace map:
\[ \langle A, B\rangle = \tr( A B). \]
Here, $\mathcal K(X)$ denotes the compact operators. For details on the general theory of operator ideals, we refer to the books \cite{Defant_Floret1992,Gohberg_Goldberg_Krupnik2000, Pietsch1980}.

We will now deal with the method of complex interpolation. For an introduction to that topic, we refer to \cite{Bergh_Lofstrom1976}, from which we also take our notation.
We want to interpolate between the spaces $\mathcal N(X)$ and $\mathcal L(X)$. Since $\mathcal N(X) \subseteq \mathcal L(X)$, we can use the method of complex interpolation to obtain new ideals between $\mathcal N(X)$ and $\mathcal L(X)$. Using $\mathcal L(X)$ as the ambient Hausdorff topological vector space, in which we embed the compatible couple $\overline{A} := (\mathcal L(X), \mathcal N(X))$ , we get 
\[ \Delta(\overline{A}) := \mathcal N(X) \cap \mathcal L(X) = \mathcal N(X) \quad \text{and} \quad \Sigma(\overline{A}) := \mathcal N(X) + \mathcal L(X) = \mathcal L(X), \]
where equalities are understood as normed vector spaces. Using the complex interpolation method, we obtain a family of subspaces of $\mathcal L(X)$:
\[ (\mathcal L(X), \mathcal N(X))_{[\theta]}, \quad 0 \leq \theta \leq 1. \]
Since $\mathcal N(X) \subseteq \mathcal L(X)$, the family of interpolation spaces is decreasing \cite[Theorem 4.2.1]{Bergh_Lofstrom1976}:
\[ \theta_0 \leq \theta_1: (\mathcal L(X), \mathcal N(X))_{[\theta_0]} \supseteq (\mathcal L(X), \mathcal N(X))_{[\theta_1]}. \]
Further, since $\Delta(\overline{A}) = \mathcal N(X)$ is dense in $(\mathcal L(X), \mathcal N(X))_{[0]}$ \cite[Theorem 4.22]{Bergh_Lofstrom1976}, we obtain using the approximation property:
\[ (\mathcal L(X), \mathcal N(X))_{[0]} = \mathcal K(X). \]
We also have
\[ (\mathcal L(X), \mathcal N(X))_{[1]} = \mathcal N(X). \]
With \cite[Theorem 4.2.2]{Bergh_Lofstrom1976} we obtain
\[ (\mathcal L(X), \mathcal N(X))_{[\theta]} = (\mathcal K(X), \mathcal N(X))_{[\theta]}, \]
i.e. each interpolation space consists of compact operators. Further, since $\mathcal L(X)$ and $\mathcal N(X)$ are ideals, for each $A \in \mathcal L(X)$ we obtain maps (which we denote by the same symbol):
\begin{align*}
L_A: \mathcal L(X) \to \mathcal L(X), \quad B \mapsto AB,\\
L_A: \mathcal N(X) \to \mathcal N(X), \quad B \mapsto AB.
\end{align*}
Interpolating this map, we obtain
\[ L_A: (\mathcal L(X), \mathcal N(X))_{[\theta]} \to (\mathcal L(X), \mathcal N(X))_{[\theta]}, \quad B \mapsto AB, \]
i.e. the interpolated spaces are left ideals. Analogously, they are right ideals. For $1 \leq p_0 < \infty$, we define the ideals of compact operators $\mathcal S^{p_0}(X)$ by
\[ S^{p_0}(X) := (\mathcal L(X), \mathcal N(X))_{[1/p_0]}. \]
In particular,
\[ \mathcal S^1(X) = \mathcal N(X). \]
One can show the norm inequalities
\[ \| A\|_{op} \leq \| A\|_{\mathcal S^{p_0}} \leq \| A\|_{\mathcal S^{q_0}} \]
for $p_0 \geq q_0$, where $\| \cdot \|_{op}$ denotes the operator norm on $\mathcal L(X)$. If $X$ is a Hilbert space, these interpolated ideals are just the usual Schatten-von Neumann ideals \cite{Pietsch_Triebel1968, Simon2005}. Surprisingly, it seems that no concrete description of the ideals $\mathcal S^{p_0}(X)$ is available if $X$ is not a Hilbert space \cite[Section 6.6.6.1]{Pietsch2007}. 

\subsection{$p$-Fock spaces}
For $t > 0$ consider the measure $\mu_t$ on $\mathbb C^n$ given by
\[ d\mu_t(z) = \frac{1}{(\pi t)^n} e^{-\frac{|z|^2}{t}} dV(z), \]
where $V$ denotes the Lebesgue measure on $\mathbb C^n$ and $|\cdot|$ is the Euclidean norm on $\mathbb C^n \cong \mathbb R^{2n}$. $\mu_t$ is well-known to be a probability measure. During the whole paper, let $1 < p < \infty$. We consider the $p$-Fock space
\[ F_t^{p} := L^p(\mathbb C^n, \mu_{2t/p}) \cap \operatorname{Hol}(\mathbb C^n). \]
Each space $F_t^{p}$ is a Banach space. The usual duality properties hold: For $q > 1$ such that $\frac{1}{p} + \frac{1}{q} = 1$ it holds $(F_t^{p})' \cong F_t^{q}$ under the usual duality pairing
\[ \langle f, g\rangle_{F_t^2} := \int_{\mathbb C^n} f(z) \overline{g(z)} d\mu_t(z), \]
and the norms of $(F_t^p)'$ and $F_t^q$ under this identification are equivalent, cf. \cite{Zhu}. When there is no confusion about which duality pairing is meant, we will write $\langle \cdot, \cdot \rangle$ instead of $\langle \cdot, \cdot \rangle_{F_t^2}$. Of course, $(F_t^2, ~\langle \cdot, \cdot\rangle_{F_t^2})$ is a Hilbert space.

For $\alpha \in \mathbb N_0^n$ we define $e_\alpha \in F_t^p$ through
\[ e_\alpha(z) = \sqrt{\frac{1}{\alpha! t^{|\alpha|}}} z^\alpha. \]
It is well-known that this is an orthonormal basis for $F_t^2$. For general $p$, it is easy to prove that this is still a Schauder basis, using Taylor expansion around the origin for each $f \in F_t^{p}$. Since every Banach space with Schauder basis has the approximation property, we obtain:
\begin{prop}\label{Proposition_reflexivity} $F_t^{p}$ is a reflexive Banach space with approximation property.
\end{prop}

In the following let $t > 0$ and $1<p<\infty$ be fixed. For $1 \leq p_0 < \infty$ we will denote by $L^{p_0}(\mathbb C^n)$ the Lebesgue space of (equivalence classes of) $p_0$-integrable functions with respect to the Lebesgue measure. $L^\infty(\mathbb C^n)$ refers to the space of measurable and essentially bounded functions on $\mathbb C^n$. For $1 \leq p_0 < \infty$ we define the Banach spaces
\begin{align*}
\mathcal A^{p_0} :&= L^{p_0}(\mathbb C^n) \oplus \mathcal S^{p_0}(F_t^{p})
\end{align*}
and
\begin{align*}
\mathcal A^\infty :&=  L^\infty(\mathbb C^n) \oplus \mathcal L(F_t^{p}),
\end{align*}
where we do not mention $t$ and $p$ in the notation for readability. These spaces can be equipped with the norms
\[ \| f \oplus A\|_{\mathcal A^{p_0}} := \max (\| f\|_{L^{p_0}}, \| A\|_{\mathcal S^{p_0}}) \]
and
\[ \| f \oplus A\|_{\mathcal A^{\infty}} := \max (\| f\|_{L^\infty}, \| A\|_{op}). \]
Note that $\mathcal A^\infty$ is a Banach algebra. It has the subalgebra
\[ \mathcal K := C_0(\mathbb C^n) \oplus \mathcal K(F_t^p), \]
where $C_0(\mathbb C^n)$ denotes the continuous functions on $\mathbb C^n$ vanishing at infinity. On $\mathcal A^1$ we have the trace map
\[ \tr(f \oplus A) := \tr(f) + \tr(A), \]
where $\tr(A)$ denotes the nuclear trace and $\tr(f) :=  \int_{\mathbb C^n} f(z) dV(z)$. As is well-known, we can identify $(L^1(\mathbb C^n))' \cong L^\infty(\mathbb C^n)$ and, as we already mentioned earlier, $(\mathcal N(X))' \cong \mathcal L(X)$ under the dual pairing induced by the trace map.

Recall that $F_t^2$ is a reproducing kernel Hilbert space with reproducing kernel
\[ K_z^t(w) := K^t(z,w) := e^{\frac{w \cdot \overline z}{t}} \]
for each $z \in \mathbb C^n$. The normalized reproducing kernels are defined through
\[ k_z^t(w) := \frac{K_z^t(w)}{\| K_z^t\|_{F_t^2}} = e^{\frac{w \cdot \overline z}{t} - \frac{1}{2t}|z|^2}. \]
It is easily seen that $k_z^t \in F_t^p$ for each $p$. For $z \in \mathbb C^n$ we define the Weyl operator $W_z \in \mathcal L(F_t^p)$ as
\[ W_z f(w) = k_z^t(w) f(w-z). \]
One can show that each $W_z$ is an isometric isomorphism and $W_z^{-1} = W_{-z}$. Further, the Banach space adjoints $W_z^\ast$ of $W_z \in \mathcal L(F_t^p)$ is given by $W_{-z} \in \mathcal L(F_t^{q})$ under the standard dual pairing (where $q$ is the exponent conjugate to $p$). For $p = 2$, these operators are unitary. One readily checks the identity
\begin{equation}\label{identity_weyl_operators}
W_z W_w = e^{-i\frac{\im(z \cdot \overline w)}{t}} W_{z+w}.
\end{equation}
With these operators, we can define an action $\alpha$ of $\mathbb C^n$ on the Banach algebras $\mathcal A^{p_0}$ and $\mathcal A^\infty$. For a measurable function $f: \mathbb C^n \to \mathbb C$ we set
\[ \alpha_z (f)(w) := f(w-z). \]
Further, for an operator $A \in \mathcal L(F_t^p)$ define
\[ \alpha_z (A) := W_z A W_{-z}. \]
$\alpha_z$ leaves $\mathcal N(F_t^p)$ invariant and therefore also $\mathcal S^{p_0}(F_t^p)$ for every $1 \leq p_0 < \infty$. Finally, for $f \oplus A$ from $\mathcal A^{p_0}$ or $\mathcal A^\infty$ we set
\[ \alpha_z (f \oplus A) := (\alpha_z f) \oplus (\alpha_z A). \]
Using Equation (\ref{identity_weyl_operators}), it is straightforward to show that
\begin{equation}\label{identity_group_action}
\alpha_z(\alpha_w(f \oplus A)) = \alpha_{z+w}(f \oplus A),
\end{equation}
which shows that $\alpha$ is indeed a group action of $\mathbb C^n$ on $\mathcal A^{p_0}$ and $\mathcal A^\infty$. Let us mention that $\| \alpha_z(A)\|_{\mathcal N} = \| A\|_{\mathcal N}$ for all $A \in \mathcal N(F_t^p)$, as one sees from the definition of the nuclear norm.
\begin{lem}\label{lemma_continuity_shifts}\begin{enumerate}[(1)]
\item Let $1 \leq p_0 < \infty$ and $f \oplus A \in \mathcal A^{p_0}$ . Then, $z \mapsto \alpha_z(f \oplus A)$ is norm-continuous on $\mathcal A^{p_0}$. 
\item Let $f \oplus A \in \mathcal K$. Then, $z \mapsto \alpha_z(f \oplus A)$ is norm-continuous in $\mathcal K$.
\item Let $f \in L^\infty(\mathbb C^n)$. Then, the map $z \mapsto \alpha_z(f)$ is weak$^\ast$ continuous.
\item Let $A \in \mathcal L(F_t^p)$. Then, the map $z \mapsto \alpha_z(A)$ is continuous with respect to the strong operator topology.
\end{enumerate}
\end{lem}
\begin{proof}
It suffices to verify continuity at $z = 0$. Norm-continuity of $z \mapsto \alpha_z(f)$ on $L^{p_0}(\mathbb C^n)$ ($1 \leq p_0 < \infty$), $C_0(\mathbb C^n)$ and weak$^\ast$ continuity on $L^\infty(\mathbb C^n)$ are well-known. Let
\[ A = y \otimes x \]
be a rank one operator ($y \in (F_t^p)', x \in F_t^p$) and identify $y$ with $\widetilde{y} \in F_t^q$ under the usual duality pairing with equivalent norms. Hence, for $f \in F_t^p$ it holds
\begin{align*}
\alpha_z (y \otimes x)(f) &= W_z(x) \int_{\mathbb C^n} W_{-z}(f)(w) \overline{\widetilde{y}(w)}d\mu_t(w)\\
&= W_z(x) \int_{\mathbb C^n} f(w) \overline{\widetilde{W_{-z}^\ast(y)}(w)} d\mu_t(w)\\
&= ((W_{-z}^\ast y) \otimes (W_z x)) (f).
\end{align*}
Now, 
\begin{align*}
\| \alpha_z(A) - A\|_{\mathcal N} &= \|  (W_{-z}^\ast y) \otimes (W_{z} x) - y \otimes x \|_{\mathcal N}\\
&\leq \| (W_{-z}^\ast y - y) \otimes (W_{z} x)\|_{\mathcal N} + \|y \otimes (W_{z} x - x)\|_{\mathcal N}\\
&\leq \| W_{-z}^\ast y - y\|_{(F_t^p)'} \| W_{z} x\|_{F_t^p} + \| y\|_{(F_t^p)'} \|W_{z} x - x\|_{F_t^p}\\
&\leq C\| W_z\widetilde{y} -\widetilde{y}\|_{F_t^q} \| W_z x\|_{F_t^p} + \| y\|_{(F_t^p)'} \| W_z x - x\|_{F_t^p}\\
&\to 0, \quad z \to 0,
\end{align*}
where we used that $W_z \to \operatorname{Id}$ strongly as $z \to 0$ on $F_t^p$ and $F_t^q$, which is easily verified. This implies $\| \alpha_z(A) - A\|_{\mathcal N} \to 0$ for all finite rank operators. Now, (1) and (2) follow through approximation by finite rank operators and the norm inequalities between the operator ideals. 

Let $A \in \mathcal L(F_t^p)$. Then, the continuity of the map $z \mapsto \alpha_z(A)$ in strong operator topology follows from the continuity of $z \mapsto W_z$ with respect to the strong operator topology.
\end{proof}
We will also need to consider the subspace of $\mathcal A^\infty$, on which the action of $\alpha$ is ``well-behaved'' in a suitable sense. We define
\begin{align*}
\mathcal C_0 &:= \{ f \in L^\infty(\mathbb C^n); ~ z \mapsto \alpha_z (f) \text{ is } \| \cdot\|_{L^\infty} \text{-continuous}\},\\
\mathcal C_1 &:= \{ A \in \mathcal L(F_t^p); ~ z \mapsto \alpha_z (A) \text{ is } \|\cdot\|_{op} \text{-continuous} \}.
\end{align*}
Set
\[ \mathcal C := \mathcal C_0 \oplus \mathcal C_1 \subset \mathcal A^\infty. \]
One can show that $\mathcal C_0 = \BUC(\mathbb C^n)$, the set of bounded and uniformly continuous functions. A precise description of $\mathcal C_1$ is not obvious, we will obtain one later. For the moment, it suffices to mention that $\mathcal C_0$ is a $C^\ast$ algebra, while $\mathcal C_1$ in general is a Banach algebra (being a $C^\ast$ algebra for $p = 2$).
\subsection{Convolutions}
In this section, we introduce the notion of convolution between objects from $\mathcal A^1$ and $\mathcal A^{p_0}$. First, we discuss how to define the convolution as a map $\ast: \mathcal A^1 \times \mathcal A^1 \to \mathcal A^1$ and derive certain properties. Afterwards, using the duality $(\mathcal A^1)' \cong \mathcal A^\infty$, we extend the convolution to a map $\mathcal A^1 \times \mathcal A^\infty \to \mathcal A^\infty$.

We follow the same path as in the Hilbert space case \cite{Luef_Skrettingland2018, Werner1984}. Some proofs follow analogously, yet several constructions and proofs need to be handled with some more care and several adaptations are necessary.

Convolution between two functions from $L^1(\mathbb C^n)$ is well known. For completeness, we define this to be
\[ f \ast g(z) := \int_{\mathbb C^n}  f(w) g(z-w) dV(w), \]
which is again in $L^1(\mathbb C^n)$. Our first goal is to extend this convolution to a convolution between two elements from $\mathcal A^1$. 
Let $f \in L^1(\mathbb C^n)$ and $A \in \mathcal N(F_t^p)$. We are going to define the convolution $f \ast A$ as a Bochner integral. We refer to the literature (e.g. \cite{Diestel_Uhl1977}) for an introduction to Bochner integration.

First, assume $f \in C_c(\mathbb C^n)$, the space of continuous functions with compact support. By Lemma \ref{lemma_continuity_shifts}, $z \mapsto f(z) \alpha_z(A)$ is continuous from $\mathbb C^n$ to $\mathcal N(F_t^p)$ and therefore weakly measurable. Further, the map is separable-valued (take e.g. $\{ f(z) \alpha_z(A); ~ z \in \mathbb Q^n \times i\mathbb Q^n\}$ as a countable dense subset of its range). Hence, the Pettis measurability theorem states that $z \mapsto f(z)\alpha_z(A)$ is strongly measurable. It is now standard to show that strong measurability carries over to $z \mapsto f(z) \alpha_z(A)$ for each $f \in L^1(\mathbb C^n)$ by approximation with $C_c(\mathbb C^n)$-functions.
Since $z \mapsto \| f(z) \alpha_z(A)\|_{\mathcal N}$ is Lebesgue integrable on $\mathbb C^n$, the function is Bochner integrable in the Banach space $\mathcal N(F_t^p)$. We define
\[ f \ast A := A \ast f := \int_{\mathbb C^n} f(z) \alpha_z(A) dV(z) \in \mathcal N(F_t^p),~ f \in L^1(\mathbb C^n), A \in \mathcal N(F_t^p). \]
It is immediate that
\[ \| f \ast A\|_{\mathcal N} \leq \int_{\mathbb C^n} |f(z)| \| \alpha_z(A)\|_{\mathcal N} dV(z) \leq \| f\|_{L^1} \| A\|_{\mathcal N}. \]
Let $U \in \mathcal L(F_t^p)$ be the isometric operator
\[ Uf(w) = f(-w). \]
For two operators $A, B \in \mathcal N(F_t^p)$, we define their convolution $A \ast B$ to be the function
\begin{equation*}
A \ast B: z \mapsto \tr(A (\alpha_z (UBU))).
\end{equation*}
\begin{lem}
Let $A, B \in \mathcal N(F_t^p)$. Then, the function $A \ast B$ is continuous and we have
\begin{equation}\label{trace_estimate}
\| A \ast B\|_{L^1} \leq C\| A\|_{\mathcal N} \| B\|_{\mathcal N}
\end{equation}
for some constant $C >0$ depending only on $n, p$ and $t$. Further,
\begin{equation}\label{equation_trace}
\tr(A \ast B) = (\pi t)^n \tr(A) \tr(B). 
\end{equation}
In particular, $A \ast B \in L^1(\mathbb C^n)$.
\end{lem}
\begin{proof}
First, observe that continuity of $A \ast B$ follows immediately from the continuity of $z \mapsto \alpha_z(UBU)$ in $\mathcal N(F_t^p)$.

Assume that $A$ and $B$ are both rank one operators. Hence,
\[ A = y_1 \otimes x_1, \quad B = y_2 \otimes x_2, \]
with $y_j \in (F_t^p)', ~x_j \in F_t^p$. We again identify $y_j$ with $\widetilde{y_j} \in F_t^q$. Then, 
\begin{align*}
A \alpha_z(UBU) &= (y_1 \otimes x_1) \alpha_z( (U^\ast y_2) \otimes (Ux_2))\\
&= (y_1 \otimes x_1) ((W_{-z}^\ast U^\ast y_2) \otimes (W_z Ux_2)).
\end{align*}
This is again a rank one operator, and one readily checks
\begin{align*}
A\alpha_z(UBU) = ((W_{-z}^\ast U^\ast y_2)\otimes x_1) \langle W_z Ux_2, y_1\rangle.
\end{align*}
Further, 
\begin{align*}
\tr(A \alpha_z(UBU)) &= \langle x_1, W_{-z}^\ast U^\ast y_2\rangle \langle W_z Ux_2, y_1\rangle\\
&= \int_{\mathbb C^n} W_{-z}(x_1)(w) \overline{U\widetilde{y_2}(w)} d\mu_t(w) \int_{\mathbb C^n} W_z U(x_2)(v) \overline{\widetilde{y_1}(v)}d\mu_t(v)\\
&= \int_{\mathbb C^n} x_1(w+z) k_{-z}^t(w) \overline{\widetilde{y_2}(-w)}d\mu_t(w) \int_{\mathbb C^n} x_2(z-v) k_z^t(v) \overline{\widetilde{y_1}(v)}d\mu_t(v).
\end{align*}
Assume for the moment that $x_j$ and $\widetilde{y_j}$ are polynomials (which are dense in $F_t^p$ and $F_t^q$, respectively). We can then apply Fubini's Theorem and obtain:
\begin{align*}
&\int_{\mathbb C^n} \tr(A \alpha_z (UBU)) dV(z)\\
= &\int_{\mathbb C^n} \overline{\widetilde{y_2}(-w)} \int_{\mathbb C^n} \overline{\widetilde{y_1}(v)} \int_{\mathbb C^n} x_1(w+z) x_2(z-v) k_{-z}^t(w) k_z^t(v) dV(z) d\mu_t(w) d\mu_t(v).
\end{align*}
Since $x_1, x_2$ are polynomials in $z_1, \dots, z_n$, they (and their product) are in $F_t^2$ as well and it holds
\begin{align*}
\int_{\mathbb C^n} &x_1(w+z) x_2(z-v) k_{-z}^t(w) k_z^t(v) dV(z)\\
 &= (\pi t)^n \int_{\mathbb C^n} x_1(w+z) x_2(z-v) e^{\frac{(v-w) \cdot \overline{z}}{t}} d\mu_t(z)\\
&= (\pi t)^n \langle x_1(w + \cdot) x_2(\cdot - v), K_{v-w}^t\rangle\\
&= (\pi t)^n x_1(v) x_2(-w).
\end{align*}
We therefore get
\begin{align*}
\int_{\mathbb C^n} &\tr(A \alpha_z(UBU)) dV(z)\\
 &= (\pi t)^n\int_{\mathbb C^n} x_1(v) \overline{\widetilde{y_1}(v)}d\mu_t(v) \int_{\mathbb C^n} x_2(-w) \overline{\widetilde{y_2}(-w)}d\mu_t(w)\\
&= (\pi t)^n \tr (y_1 \otimes x_1) \tr(y_2 \otimes x_2).
\end{align*}
Letting $x = x_1 = \widetilde{y_2}$, $y = x_2 = \widetilde{y_1}$ (which is possible, since we still assume that they are polynomials) we have
\begin{align*}
\int_{\mathbb C^n} |\langle y, W_{z}Ux\rangle|^2 dV(z) = (\pi t)^n |\tr( y \otimes x)|^2  
\end{align*}
and hence it holds $\langle y , W_z U x\rangle \in L^2(\mathbb C^n)$ as a function of $z$ with
\[ \| \langle y, W_z Ux\rangle \|_{L^2} \leq (\pi t)^{n/2} |\tr (y \otimes x)| \leq (\pi t)^{n/2} \| y\|_{(F_t^p)'} \| x\|_{F_t^p}. \]
Therefore,
\[ \tr(A \alpha_z(UBU)) = \langle W_{-z}^\ast U^\ast y_2, x_1\rangle \langle y_1, W_z Ux_2\rangle \in L^1(\mathbb C^n) \]
(understood as a function of $z$) and H\"{o}lder's inequality yields the estimate
\[ \| \tr(A \alpha_z (UBU))\|_{L^1} \leq C \| y_1\|_{(F_t^p)'} \| y_2\|_{(F_t^p)'} \| x_1\|_{F_t^p} \| x_2\|_{F_t^p} \]
for some constant $C$ depending only on $n, t$ and $p$. 
Now, let $x_j \in F_t^p$, $y_j \in (F_t^p)'$ be arbitrary. Let $(x_j^m)_m$, $(\widetilde{y_j^m})_m$ be sequences of polynomials converging to $x_j$ and $\widetilde{y_j}$ in $F_t^p$ and $F_t^q$, respectively. Then,
\begin{align*}
\tr(A\alpha_z(UBU)) &= \langle W_{-z}^\ast U^\ast y_2, x_1\rangle \langle y_1, W_z Ux_2\rangle\\
&= \lim_{m \to \infty} \langle W_{-z}^\ast U^\ast y_2^m, x_1^m\rangle \langle y_1^m, W_z Ux_2^m\rangle.
\end{align*}
By Fatou's Lemma we get
\[ \| \tr(A\alpha_z(UBU))\|_{L^1} \leq C \| y_1\|_{(F_t^p)'} \| y_2\|_{(F_t^p)'} \| x_1\|_{F_t^p} \| x_2\|_{F_t^p}. \]
Now, taking the infimum over all possible representations (\ref{nuclear_operator}) gives
\[ \| \tr(A \alpha_z(UBU))\|_{L^1} \leq C \| A\|_{\mathcal N} \| B\|_{\mathcal N} \]
for arbitrary rank one operators. Having this estimate, it is easy to derive Equation (\ref{equation_trace}) for arbitrary rank one operators. Finally, it is standard to generalize (\ref{trace_estimate}) and (\ref{equation_trace}) from rank one operators to arbitrary nuclear operators.
\end{proof}

Combining the last few results, we see that we obtain (by linear extension) a convolution
\[ \ast: \mathcal A^1 \times \mathcal A^1 \to \mathcal A^1. \]

\begin{lem}\label{properties_convolution}
The convolution is commutative and associative. Further, for $f_1, f_2 \in L^1(\mathbb C^n)$, $A_1, A_2 \in \mathcal N(F_t^p)$ we have
\begin{align*}
\alpha_z (f_1 \ast f_2) &= \alpha_z(f_1) \ast f_2,\\
\alpha_z (f_1 \ast A_2) &= \alpha_z(f_1) \ast A_2 = f_1 \ast \alpha_z(A_2),\\
\alpha_z(A_1 \ast A_2) &= \alpha_z(A_1) \ast A_2.
\end{align*}
and there is a constant $C>0$ (depending on $n$, $p$ and $t$) such that
\[ \| (f_1 \oplus A_1) \ast (f_2 \oplus A_2)\|_{\mathcal A^1} \leq C \| f_1 \oplus A_1\|_{\mathcal A^1} \| f_2 \oplus A_2\|_{\mathcal A^1}. \]
\end{lem}
\begin{proof}
The proof of commutativity, associativity and the three identities carries over from the Hilbert space proof \cite{Luef_Skrettingland2018} and follows from properties of the trace functional and the Bochner integral. Here, we present only the associativity for the convolution of three operators - this is the most difficult statement of all those listed above and will be a key fact later on. The computations presented here are essentially the same as in the proof of \cite[Proposition 4.4]{Luef_Skrettingland2018}.

Recall that $\mathcal L(F_t^p) \cong \mathcal N(F_t^p)'$ by identifying each operator $D \in \mathcal L(F_t^p)$ with the linear functional $\phi_D(N) = \tr(ND)$. In particular, two operators $D_1, D_2 \in \mathcal L(F_t^p)$ are identical if and only if $\tr(ND_1) = \tr(ND_2)$ for every $N \in \mathcal N(F_t^p)$.

Let $A, B, C, N \in \mathcal N(F_t^p)$. Then, $(A \ast B) \ast C,~ A \ast (B \ast C) \in \mathcal N(F_t^p)$. We compare $\tr(N((A \ast B) \ast C))$ with $\tr(N(A \ast (B \ast C)))$:
\begin{align*}
\operatorname{Tr}(N ((B \ast C) \ast A)) = &\operatorname{Tr}(N((C \ast B) \ast A)
\\
= &\operatorname{Tr}\left ( N \int _{\mathbb{C}^{n}}
\operatorname{Tr}(CW_{z} U BU W_{-z}) W_{z} A W_{-z} ~dV(z) \right )
\\
= &\int _{\mathbb{C}^{n}} \operatorname{Tr}(W_{z} A W_{-z} N)
\operatorname{Tr}(CW_{z} U B U W_{-z}) ~dV(z)
\end{align*}
Since $UU = \operatorname{Id}$ we have $\tr(CW_z U BU W_{-z}) = \tr(UCW_z U B U W_{-z} U)$. Applying Equation \eqref{equation_trace} yields
\begin{align*}
= &\frac{1}{(\pi t)^{n}} \int _{\mathbb{C}^{n}} \int _{\mathbb{C}^{n}}
\operatorname{Tr}(W_{z} A W_{-z} N W_{w} C W_{z} U B U W_{-z} W_{-w} )~dV(w)
~dV(z)
\\
= &\frac{1}{(\pi t)^{n}} \int _{\mathbb{C}^{n}} \int _{\mathbb{C}^{n}}
\operatorname{Tr}(W_{z} A W_{-z} N W_{w} C W_{-w} W_{w} W_{z} U B U W_{-z}
W_{-w} )~dV(w) ~dV(z)
\end{align*}
Since $\alpha_z(\alpha_w(UBU)) = \alpha_w(\alpha_z(UBU))$:
\begin{align*}
= &\frac{1}{(\pi t)^{n}} \int _{\mathbb{C}^{n}} \int _{\mathbb{C}^{n}}
\operatorname{Tr}(W_{z} A W_{-z} N W_{w} C W_{-w} W_{z} W_{w} U B U W_{-w}
W_{-z} )~dV(w) ~dV(z)
\\
=&\frac{1}{(\pi t)^{n}} \int _{\mathbb{C}^{n}} \int _{\mathbb{C}^{n}}
\operatorname{Tr}(N W_{w} C W_{-w} W_{z} W_{w} U B U W_{-w} W_{-z} W_{z}
A W_{-z})~dV(w) ~dV(z)
\\
=&\frac{1}{(\pi t)^{n}} \int _{\mathbb{C}^{n}} \int _{\mathbb{C}^{n}}
\operatorname{Tr}(N W_{w} C W_{-w} W_{z} W_{w} U B U W_{-w} A W_{-z})~dV(w)
~dV(z)
\end{align*}
Using Estimate \eqref{trace_estimate} one can show that Fubini's Theorem applies here, which gives, together with another application of Equation \eqref{equation_trace}:
\begin{align*}
= &\frac{1}{(\pi t)^{n}} \int _{\mathbb{C}^{n}} \int _{\mathbb{C}^{n}}
\operatorname{Tr}(N W_{w} C W_{-w} W_{z} W_{w} U B U W_{-w} A W_{-z})~dV(z)
~dV(w)
\\
= &\int _{\mathbb{C}^{n}} \operatorname{Tr}(NW_{w} CW_{-w})
\operatorname{Tr}(W_{w} U B U W_{-w} A)~dV(w)
\\
= &\operatorname{Tr}(N ((A \ast B)\ast C))
\end{align*}

\end{proof}
We will denote by
$\langle f, g\rangle_{\text{tr}},~\langle A, B\rangle_{\text{tr}}$
the duality pairing induced by the trace maps, i.e. for $f \in L^1(\mathbb C^n),~ g \in L^\infty(\mathbb C^n)$ and $A \in \mathcal N(F_t^p),~ B \in \mathcal L(F_t^p)$ we have
\[ \langle f, g\rangle_{\text{tr}} = \int_{\mathbb C^n} f(z) g(z) dV(z), \quad \langle A, B\rangle_{\text{tr}} = \tr(AB). \]
Our next goal will be to extend the convolution such that we can convolve elements from $\mathcal A^1$ with elements from $\mathcal A^\infty$. The following identities will be useful for this:
\begin{lem}\label{Lemma_adjoint_maps}
Let $f \in L^1(\mathbb C^n)$ and $A_1, A_2 \in \mathcal N(F_t^p)$. Then, we have
\begin{align*}
\langle f \ast A_1, B\rangle_{\emph{tr}} &= \langle f, B \ast (UA_1U)\rangle_{\emph{tr}}, \quad B \in \mathcal N(F_t^p),\\
\langle f \ast A_2, B\rangle_{\emph{tr}} &= \langle A_2, (Uf) \ast B\rangle_{\emph{tr}}, \quad B \in \mathcal N(F_t^p),\\
\langle A_1 \ast A_2, g\rangle_{\emph{tr}} &= \langle A_1, g \ast (UA_2 U)\rangle_{\emph{tr}}, \quad g \in L^1(\mathbb C^n).
\end{align*}
\end{lem}
\begin{proof}
Follows again from simple properties of the trace map and the Bochner integral.
\end{proof}
We now want to extend the convolution $\ast: \mathcal A^1 \times \mathcal A^1 \to \mathcal A^1$ to a larger class, at least in the second factor. In the end, we will obtain a convolution $\ast: \mathcal A^1 \times \mathcal A^\infty \to \mathcal A^\infty$. For $f_1 \in L^1(\mathbb C^n)$, $A_1 \in \mathcal N(F_t^p)$ and $A_2 \in \mathcal L(F_t^p)$, we could still define the convolutions as
\begin{align*}
f_1 \ast A_2 &= \int_{\mathbb C^n} f_1(z) W_z A_2 W_{-z} dV(z), \\
A_1 \ast A_2(z) &= \tr( A_1 \alpha_z(U A_2 U)).
\end{align*}
These are actually equivalent to the definitions that we will give below (the integral above has to be understood as a weak$^\ast$ integral), but for $A_1 \ast f_2$ ($f_2 \in L^\infty(\mathbb C^n), ~A_1 \in \mathcal N(F_t^p)$) the analogous definition of $A_1 \ast f_2$ makes no sense right away. The problem can be solved by defining $A_1 \ast f_2$ via duality $\mathcal L(F_t^p) \cong (\mathcal N(F_t^p))'$. To have a unified approach, we define the convolution in all three cases through duality:
\begin{defn}
Let $f_1 \in L^1(\mathbb C^n), ~f_2 \in L^\infty(\mathbb C^n)$, $A_1 \in \mathcal N(F_t^p)$ and $A_2 \in \mathcal L(F_t^p)$. We define the convolutions $f_1 \ast A_2 \in \mathcal L(F_t^p)$, $f_2 \ast A_1 \in \mathcal L(F_t^p)$ and $A_1 \ast A_2 \in L^\infty(\mathbb C^n)$ through the following duality relations:
\begin{align*}
\langle f_1 \ast A_2, B\rangle_{\text{tr}} &= \langle A_2, Uf_1 \ast B\rangle_{\text{tr}} \quad &\text{ for all } B \in \mathcal N(F_t^p)\\
\langle f_2 \ast A_1, B\rangle_{\text{tr}} &= \langle f_2, B \ast (UA_1 U)\rangle_{\text{tr}} \quad &\text{ for all } B \in \mathcal N(F_t^p)\\
\langle A_1 \ast A_2, g\rangle_{\text{tr}} &= \langle A_2, g \ast (UA_1U)\rangle_{\text{tr}} \quad &\text{ for all } g \in L^1(\mathbb C^n)
\end{align*}
We extend this convolution linearly to a map $\ast: \mathcal A^1 \times \mathcal A^\infty \to \mathcal A^\infty$. 
\end{defn}
\begin{lem}\label{lemma_convolution_ainfty}\begin{enumerate}[(1)]
\item For $(f_1 \oplus A_1) \in \mathcal A^1$ and $(f_2 \oplus A_2) \in \mathcal A^\infty$ it holds
\[ \| (f_1 \oplus A_1) \ast (f_2 \oplus A_2)\|_{\mathcal A^\infty} \leq C \| (f_1 \oplus A_1)\|_{\mathcal A^1} \| (f_2 \oplus A_2) \|_{\mathcal A^\infty} \]
for some constant $C > 0$ which depends only on $n, ~p$ and $t$.
\item The convolution $\ast: \mathcal A^1 \times \mathcal A^{\infty} \to \mathcal A^\infty$ is associative in the sense that for $(f_1 \oplus A_1), (f_2 \oplus A_2) \in \mathcal A^1$ and $(g \oplus B) \in \mathcal A^\infty$ we have
\[ [(f_1 \oplus A_1) \ast (f_2 \oplus A_2)] \ast (g \oplus B) = (f_1 \oplus A_1) \ast [(f_2 \oplus A_2) \ast (g \oplus B)]. \]
\item Let $(f_1 \oplus A_1) \in \mathcal A^1$ and $(f_2 \oplus A_2) \in \mathcal A^\infty$. Then, the following identities are valid:
\begin{align*}
\alpha_z (f_1 \ast f_2) = \alpha_z(f_1) &\ast f_2 = f_1 \ast \alpha_z(f_2),\\
\alpha_z (f_1 \ast A_2) = \alpha_z(f_1) &\ast A_2 = f_1 \ast \alpha_z(A_2),\\
\alpha_z (A_1 \ast f_2) = \alpha_z(A_1) &\ast f_2 = A_1 \ast \alpha_z(f_2),\\
\alpha_z(A_1 \ast A_2) = \alpha_z(A_1) &\ast A_2 = A_1 \ast \alpha_z(A_2).
\end{align*}
\end{enumerate}
\end{lem}
\begin{proof}
The result follows immediately from the duality relations, which define the convolution, and corresponding relations for the convolution between $\mathcal A^1$ objects. We prove only one particular associativity statement, since this will be important later.

Let $A, B \in \mathcal N(F_t^p)$ and $C \in \mathcal L(F_t^p)$. It is no problem to verify the identity
\begin{equation*}
(UAU) \ast (UBU) = U(A \ast B).
\end{equation*}
In particular, for any $N \in \mathcal N(F_t^p)$ we obtain
\begin{equation*}
(N \ast (UAU)) \ast (UBU) = N \ast ((UAU) \ast (UBU)) = N \ast (U(A \ast B)).
\end{equation*}
Hence,
\begin{align*}
\langle A \ast (B \ast C), N\rangle_{\text{tr}} &= \langle B \ast C, N \ast (UAU)\rangle_{\text{tr}}\\
&= \langle C, (N \ast (UAU)) \ast (UBU) \rangle_{\text{tr}}\\
&= \langle C, N \ast (U(A \ast B))\rangle_{\text{tr}}\\
&= \langle (A \ast B) \ast C, N\rangle_{\text{tr}}
\end{align*}
\end{proof}
Since $\mathcal A^{p_0}$ was set up such that it is an interpolation space between $\mathcal A^1$ and $\mathcal A^\infty$, we can define the convolution as a map $\mathcal A^1 \times \mathcal A^{p_0} \to \mathcal A^{p_0}$ through interpolation and all the properties of the previous lemma carry over with obvious modifications.

We have the following important observation:
\begin{lem}\label{Lemma_continuity}
Let $(f_1 \oplus A_1) \in \mathcal A^1$, $(f_2 \oplus A_2) \in \mathcal A^\infty$. Then, we have
\[ (f_1 \oplus A_1) \ast (f_2 \oplus A_2) \in \mathcal C, \]
i.e. $f_1 \ast f_2, ~A_1 \ast A_2 \in \BUC(\mathbb C^n)$ and $f_1 \ast A_2, ~ f_2 \ast A_1 \in \mathcal C_1$.
\end{lem}
\begin{proof}
Using Lemma \ref{lemma_convolution_ainfty}(3) we have
\begin{align*}
\| \alpha_z(f_1 \ast A_2) - \alpha_w(f_1 \ast A_2)\|_{op} &= \| (\alpha_z(f_1) - \alpha_w(f_1)) \ast A_2\|_{op}\\
&\leq \| \alpha_z(f_1) - \alpha_w(f_1)\|_{L^1} \| A_2\|_{op}\\
&\to 0, \quad z \to w,
\end{align*}
where we used that $\alpha$ acts continuously on $L^1(\mathbb C^n)$. The other cases are proven analogously.
\end{proof}
Although the convolutions are now defined through duality, we might still use their old definitions in two of the three cases, as we already mentioned above. For simplicity, we discuss the case $f_1 \ast A_2$ for $f_1 \in L^1(\mathbb C^n)$ only if $A_2 \in \mathcal C_1$, which gives some extra information.
\begin{lem}\begin{enumerate}[1)]
\item Let $f_1 \in L^1(\mathbb C^n)$ and $A_2 \in \mathcal C_1$. Then, their convolution can be expressed as
\[ f_1 \ast A_2 = \int_{\mathbb C^n} f_1(z) \alpha_z(A_2) dV(z) \in \mathcal C_1. \]
\item For $A_1 \in \mathcal N(F_t^p)$ and $A_2 \in \mathcal L(F_t^p)$, we have
\[ A_1 \ast A_2(z) = \tr(A_1 \alpha_z(UA_2U)). \]
\end{enumerate}
\end{lem}
\begin{proof}
Using properties of Bochner integrals and trace maps, one verifies that these objects satisfy the duality definition of the convolutions. In 1) $f_1 \ast A_2 \in \mathcal C_1$ follows, since the integral converges in $\mathcal C_1$ as a Bochner integral.
\end{proof}

\subsection{Toeplitz quantization, Berezin transform and convolution}
Recall that for $p = 2$, the orthogonal projection
\[ P_t: L^2(\mathbb C^n, \mu_t) \to F_t^2 \]
is given by
\[ P_t(f)(z) = \int_{\mathbb C^n} e^{\frac{z \cdot \overline{w}}{t}} f(w) d\mu_t(w). \]
As is well-known, the operator defined by the same integral expression defines a continuous projection
\begin{align*}
P_t: ~&L^p(\mathbb C^n, \mu_{2t/p}) \to F_t^p,\\
P_t(f)(z) &= \int_{\mathbb C^n} e^{\frac{z \cdot \overline w}{t}} f(w) d\mu_t(w)\\
&= \left( \frac{2}{p} \right)^n \int_{\mathbb C^n} e^{\frac{z\cdot \overline w}{t}}f(w) e^{(\frac{p}{2t} - \frac{1}{t})|w|^2} d\mu_{2t/p}(w),
\end{align*}
cf. \cite{Janson_Peetre_Rochberg, Zhu}. For $f \in L^\infty(\mathbb C^n)$ we will denote by $T_f^t$ (suppressing $p$ in the notation) the Toeplitz operator
\[ T_f^t: F_t^p \to F_t^p, \quad T_f^tg = P_t(fg). \]
It is easy to see that $T_f^t$ is bounded if $f \in L^\infty(\mathbb C^n)$. Further, for $f \in L^1(\mathbb C^n)$ we define $T_f^t$ by the same formula. Here, boundedness is not entirely trivial. For a suitable measurable function $f: \mathbb C^n \to \mathbb C$ we will define its Berezin transform at $t > 0$ through
\[ \widetilde{f}^{(t)}(z) := \langle f k_z^t, k_z^t\rangle_{F_t^2}, \]
if it exists. Further, for $A \in \mathcal L(F_t^p)$ we define the Berezin transform as
\[ \widetilde{A}(z) := \langle Ak_z^t, k_z^t\rangle_{F_t^2}. \]
For $f \in L^\infty(\mathbb C^n)$, one readily checks that
\[ \widetilde{T_f^t} = \widetilde{f}^{(t)}. \]
It is our next goal to express the maps $f \mapsto T_f^t$ and $A \mapsto \widetilde{A}$ using convolutions. For this, we consider the operator $P_{\mathbb C} = 1 \otimes 1 \in \mathcal N(F_t^p)$, i.e.
\[ P_{\mathbb C}f = f(0) \in F_t^p.\]
We will also need the following normalized version of $P_{\mathbb C}$:
\[ \mathcal R_t := \frac{1}{(\pi t)^n} P_{\mathbb C}. \]
\begin{lem}
Let $A \in \mathcal L(F_t^p)$. Then, we have
\[ \widetilde{A} = P_{\mathbb C} \ast A. \]
\end{lem}
\begin{proof}
One readily checks for $f \in F_t^p$ that
\[ AW_z U P_{\mathbb C} U W_{-z}(f) = Ak_z^t \cdot f(z) e^{-\frac{|z|^2}{2t}}. \]
Therefore, an eigenvector of that operator to a non-zero eigenvalue needs to be a multiple of $Ak_z^t$, and for these eigenvectors we obtain
\[ AW_z U P_{\mathbb C} U W_{-z}(Ak_z) = Ak_z^t \cdot Ak_z^t(z) e^{-\frac{|z|^2}{2t}}, \]
i.e. the only non-zero eigenvalue of $AW_z U P_{\mathbb C} U W_{-z}$ is
\[ Ak_z^t(z) e^{-\frac{|z|^2}{2t}} = \langle Ak_z^t, K_z^t\rangle_{F_t^2} e^{-\frac{|z|^2}{2t}} = \widetilde{A}(z). \]
Since the trace of a finite rank operator coincides with the sum of its eigenvalues, the proof is finished.
\end{proof}
\begin{lem}
For $f \in L^1(\mathbb C^n)$ it holds $T_f^t = \mathcal R_t \ast f \in \mathcal N(F_t^p)$. In particular, $T_f^t$ is a bounded linear operator.
\end{lem}
\begin{proof}
For $f \in L^1(\mathbb C^n)$ we have
\begin{align*}
P_{\mathbb C} \ast f (g) &= \int_{\mathbb C^n} f(z) W_{z} P_{\mathbb C} W_{-z} g ~ dV(z)\\
 &= \int_{\mathbb C^n} f(z)W_{z}(1) e^{-\frac{|z|^2}{2t}} g(z) dV(z)\\
 &= \int_{\mathbb C^n} f(z) e^{\frac{\langle \cdot, z\rangle}{t}} g(z) e^{-\frac{|z|^2}{t}} dV(z)\\
 &= (\pi t)^n T_f^t g.
\end{align*}
\end{proof}
The next goal is to extend the relation $\mathcal R_t \ast f = T_f^t$ to all $f \in L^\infty(\mathbb C^n)$.
\begin{prop}\label{Proposition_Toeplitz_identity}
Let $f \in L^\infty(\mathbb C^n)$. Then, we have $\mathcal R_t \ast f= T_f^t$.
\end{prop}
\begin{proof}
Recall that for $f \in L^\infty(\mathbb C^n)$ the convolution $\mathcal R_t \ast f$ was defined through the duality relation
\[ \langle \mathcal R_t \ast f, B\rangle_{\text{tr}} = \langle f, B \ast (U \mathcal R_t U)\rangle_{\text{tr}}, \quad B \in \mathcal N(F_t^p). \]
It is simple to check that $U P_{\mathbb C} U = P_{\mathbb C}$. Hence, we obtain
\[ \langle \mathcal R_t \ast f, B\rangle_{\text{tr}} = \langle f, B \ast \mathcal R_t\rangle_{\text{tr}} = \frac{1}{(\pi t)^n}\langle f, \widetilde{B}\rangle_{\text{tr}}. \]
Letting $B = k_z^t \otimes k_z^t$, we obtain
\[ \langle \mathcal R_t \ast f, k_z^t \otimes k_z^t\rangle_{\text{tr}} = \frac{1}{(\pi t)^n}\langle f, (k_z^t \otimes k_z^t)^\sim\rangle_{\text{tr}}. \]
For each $A \in \mathcal L(F_t^p)$, the operator $A (k_z^t \otimes k_z^t)$ acts as
\[ A (k_z^t \otimes k_z^t)(g) = Ak_z^t \cdot \langle g, k_z^t\rangle_{F_t^2}, \]
i.e. we obtain for the trace
\[ \langle A, k_z^t \otimes k_z^t\rangle_{\text{tr}} = \tr(A (k_z^t \otimes k_z^t)) = \langle A k_z^t, k_z^t\rangle_{F_t^2} = \widetilde{A}(z). \]
Further, the Berezin transform of $k_z^t \otimes k_z^t$ is given by
\begin{align*}
(k_z^t \otimes k_z^t)^\sim(w) = \langle (k_z^t \otimes k_z^t)(k_w^t), k_w^t\rangle_{F_t^2} = \langle k_z^t, k_w^t\rangle_{F_t^2} \langle k_w^t, k_z^t\rangle_{F_t^2} = e^{-\frac{|z-w|^2}{t}},
\end{align*}
which yields
\[ \langle f, (k_z^t \otimes k_z^t)^\sim\rangle_{\text{tr}} = \int_{\mathbb C^n}f(w) e^{-\frac{|z-w|^2}{t}}dV(w) = (\pi t)^n \widetilde{f}^{(t)}(z). \]
Therefore, the operator $\mathcal R_t \ast f$ fulfils the relation
\[ (\mathcal R_t \ast f)^\sim = \widetilde{f}^{(t)}, \]
which then implies that $\mathcal R_t \ast f= T_f^t$ as the Berezin transform is injective.
\end{proof}
Let us consider the maps
\[ \Psi: L^1(\mathbb C^n) \to \mathcal N(F_t^p), \quad \Psi(f) = T_f^t = \mathcal R_t \ast f\]
and
\[ \Phi: \mathcal N(F_t^p) \to L^1(\mathbb C^n), \quad \Phi(A) = \widetilde{A} = P_{\mathbb C} \ast A. \]
Restating the duality relations for the convolutions, we obtain for $f \in L^1(\mathbb C^n)$ and $B \in \mathcal L(F_t^p) \cong (\mathcal N(F_t^p))'$:
\begin{align*}
\langle \Psi(f), B\rangle_{\text{tr}} = \langle \mathcal R_t \ast f, B\rangle_{\text{tr}} = \langle f, (U \mathcal R_t U) \ast B \rangle_{\text{tr}} = \langle f, \mathcal R_t \ast B\rangle_{\text{tr}},
\end{align*}
i.e. the Banach space adjoint of $\Psi$ is given by
\[ (\Psi)': \mathcal L(F_t^p) \to L^\infty(\mathbb C^n), \quad (\Psi)'(A) = \frac{1}{(\pi t)^n}\widetilde{A}. \]
Analogously, one sees that for $A \in \mathcal N(F_t^p)$ and $g \in L^\infty(\mathbb C^n)$ we have
\begin{align*}
\langle \Phi(A), g\rangle_{\text{tr}} &= \langle P_{\mathbb C} \ast A, g\rangle_{\text{tr}} = \langle A, (U P_{\mathbb C} U) \ast g\rangle_{\text{tr}} = (\pi t)^n\langle A, T_g^t\rangle_{\text{tr}},
\end{align*}
i.e.
\[ (\Phi)': L^\infty(\mathbb C^n) \to \mathcal L(F_t^p), \quad (\Phi)'(f) = (\pi t)^n T_f^t. \]
\begin{prop}\label{proposition_density}
$\{ T_f^t; ~ f \in L^1(\mathbb C^n)\}$ is dense in $\mathcal N(F_t^p)$ and $\{ \widetilde{A}; ~ A \in \mathcal N(F_t^p)\}$ is dense in $L^1(\mathbb C^n)$.
\end{prop}
\begin{proof}
As is well-known, the Berezin transform $A \mapsto \widetilde{A}$ ($A \in \mathcal L(F_t^p)$) is injective. Further, the map $f \mapsto T_f^t$ ($f \in L^\infty(\mathbb C^n)$) is also injective: If $T_f^t = 0$, then also $\widetilde{T_f^t} = \widetilde{f}^{(t)} = 0$. Since $\widetilde{f}^{(t)}$ is just the heat transform of $f$ at time $t/4$, the map $f \mapsto \widetilde{f}^{(t)}$ is well-known to be injective \cite[Proposition 3.17]{Zhu}. Hence, $T_f^t = 0$ implies $f = 0$.

Since the above two maps are injective, the Banach space adjoints $(\Psi)'$ and $(\Phi)'$ of the maps $\Psi$ and $\Phi$ are injective. This in turn implies that $\Psi$ and $\Phi$ have dense range, which is just the statement of the proposition.
\end{proof}
We have seen that convolution by $\mathcal R_t$ yields the maps
\begin{align*}
L^1(\mathbb C^n) \to \mathcal N(F_t^p), \quad f &\mapsto T_f^t\\
L^\infty(\mathbb C^n) \to \mathcal L(F_t^p), \quad f &\mapsto T_f^t.
\end{align*}
Further, convolution by $P_{\mathbb C}$ yields maps
\begin{align*}
\mathcal N(F_t^p) \to L^1(\mathbb C^n), \quad A &\mapsto \widetilde{A}\\
\mathcal L(F_t^p) \to L^\infty(\mathbb C^n), \quad A &\mapsto \widetilde{A}.
\end{align*}
Applying now complex interpolation to both maps, we obtain the following result. At least part (i) is already well-known in the case $p = 2$ with different proof, compare e.g. \cite{Isralowitz_Zhu2010, Zhu}.
\begin{lem}Let $1 \leq p_0 < \infty$. \begin{enumerate}[(i)]
\item For $f \in L^{p_0}(\mathbb C^n)$ we have $T_f^t \in \mathcal S^{p_0}(F_t^p)$. 
\item For $A \in \mathcal S^{p_0}(F_t^p)$ we have $\widetilde{A} \in L^{p_0}(\mathbb C^n)$.
\end{enumerate}
\end{lem}
Simple approximation arguments yield now the following:
\begin{lem}\label{lemma_sp_density}Let $1 \leq p_0 < \infty$. 
\begin{enumerate}[(i)]
\item $\{ T_f^t; ~ f \in L^{p_0}(\mathbb C^n)\}$ is dense in $\mathcal S^{p_0}(F_t^p)$.
\item $\{ \widetilde{A}; ~ A \in \mathcal S^{p_0}(F_t^p)\}$ is dense in $L^{p_0}(\mathbb C^n)$.
\end{enumerate}
\end{lem}

\subsection{Characterizations of $\mathcal C_1$}\label{section_approximate_identity}
In the following, we will denote by $f_s$ the function
\[ f_s(z) = \frac{1}{(\pi s)^n}e^{-\frac{|z|^2}{s}}, \]
where $s > 0$. The result of this section is the following:
\begin{prop}\label{proposition_characterizations_C_1}
The following equalities hold true:
\begin{align*}
\mathcal C_1 &= \overline{\Span} \{ g \ast B; ~g \in L^1(\mathbb C^n), ~B \in \mathcal L(F_t^p)\}\\
&= \{ B \in \mathcal L(F_t^p); ~ f_s \ast B \to B \text{ in operator norm as } s \to 0 \}\\
&= \overline{\mathcal R_t \ast \BUC(\mathbb C^n)}\\
&= \{ g \ast B; ~ g \in L^1(\mathbb C^n), ~ B \in \mathcal C_1\},
\end{align*}
where closures are taken in $\mathcal L(F_t^p)$.
\end{prop}
\begin{rem}
Once we have proven the first three equalities, the last equality follows directly from the Cohen-Hewitt factorization theorem. Since we will not need this equality, we do not discuss the factorization theorem and refer to the literature (e.g. \cite{Doran_Wichmann1979}). The equality
\[ \mathcal C_1 = \{ B \in \mathcal L(F_t^p); ~ g_s \ast B \to B \text{ in operator norm as } s \to 0\} \]
is well-known in the theory of Banach modules over locally compact groups. We give a short proof below for completeness.
\end{rem}
We prove the following lemma as a preparation.
\begin{lem}\label{Proposition_approximate_identity}
Convolution by $f_s$ is an approximate identity in $\mathcal A^{p_0}$ for $1 \leq p_0 < \infty$ and in $\mathcal C$, i.e.
\begin{align*}
\| f_s \ast (g_1 \oplus A_1) - (g_1 \oplus A_1)\|_{\mathcal A^{p_0}} &\to 0, \quad s \to 0\\
\| f_s \ast (g_2 \oplus A_2) - (g_2 \oplus A_2) \|_{\mathcal A^\infty} &\to 0, \quad s \to 0
\end{align*}
for each $(g_1 \oplus A_1) \in \mathcal A^{p_0}$ and $(g_2 \oplus A_2) \in \mathcal C$.
\end{lem}
\begin{proof}
For $g_1 \in L^{p_0}(\mathbb C^n)$ and $g_2 \in \BUC(\mathbb C^n)$ it is well-known and not hard to prove that
\begin{align*}
\| &f_s \ast g_1 - g_1\|_{L^{p_0}} \to 0, \quad s \to 0,\\
\| &f_s \ast g_2 - g_2\|_{L^\infty} \to 0, \quad s \to 0.
\end{align*}
Consider operators of the form
\[ A_1 = \mathcal R_t \ast g_1 \]
with, as above, $g_1 \in L^{p_0}(\mathbb C^n)$. For $g \in L^1(\mathbb C^n)$, one easily establishes the identity 
\begin{align*}
f_s \ast (\mathcal R_t \ast g) = \mathcal R_t \ast (f_s \ast g)
\end{align*}
using Lemma \ref{properties_convolution}, which then carries over to the case $g \in L^{p_0}(\mathbb C^n)$. We therefore obtain
\begin{align*}
\| f_s \ast (\mathcal R_t \ast g_1) - \mathcal R_t \ast g_1\|_{\mathcal S^{p_0}} &= \| \mathcal R_t \ast (f_s \ast g_1 - g_1) \|_{\mathcal S^{p_0}}\\
&\leq \| \mathcal R_t\|_{\mathcal N} \| f_s \ast g_1 - g_1\|_{L^{p_0}}\\
&\to 0, \quad s \to 0
\end{align*}
by the $\mathcal A^{p_0}$-version of Lemma \ref{lemma_convolution_ainfty}. By Lemma \ref{lemma_sp_density}(i), $\mathcal R_t \ast L^{p_0}(\mathbb C^n) $ is dense in $\mathcal S^{p_0}(F_t^p)$, hence the result for $\mathcal S^{p_0}(F_t^p)$ follows from some standard density argument.

Let $B \in \mathcal C_1$, i.e. $z \mapsto W_z B W_{-z}$ is continuous with respect to the operator norm. We claim that $f_s \ast B \to B$ in operator norm as $s \to 0$. Using basic properties of the Bochner integral we obtain
\begin{align*}
\| B - f_s \ast B\|_{op} &= \left \| \int_{\mathbb C^n} f_s(z) B - f_s(z) W_z B W_{-z} dV(z) \right \|_{op}\\
&\leq \int_{\mathbb C^n} f_s(z) \| B - W_z B W_{-z}\|_{op} ~dV(z).
\end{align*}
An easy consequence of $B \in \mathcal C_1$ and the inverse triangle inequality is the fact that $z \mapsto \| B - W_z B W_{-z}\|$ is in $\BUC(\mathbb C^n)$. Therefore, we have
\[ \int_{\mathbb C^n} f_s(z) \| B - W_z B W_{-z}\|_{op} ~dV(z) \to \| B - W_0 B W_{-0}\|_{op} = 0, \quad s \to 0, \]
and thus $f_s \ast B \to B$.
\end{proof}
\begin{proof}[Proof of Proposition \ref{proposition_characterizations_C_1}]
The following inclusions hold true due to Lemma \ref{Lemma_continuity}:
\begin{align*}
\mathcal C_1 &\supseteq \overline{\Span} \{ g \ast B; ~g \in L^1(\mathbb C^n), ~B \in \mathcal L(F_t^p)\}\\
&\supseteq \{ B \in \mathcal L(F_t^p); ~f_s \ast B \to B \text{ in operator norm as } s \to 0\}.
\end{align*}
The previous lemma proves the inclusion
\[ \mathcal C_1 \subseteq \{ B \in \mathcal L(F_t^p); ~ f_s \ast B \to B \text{ in operator norm as } s \to 0 \}, \]
i.e. we obtain
\begin{align*}
\mathcal C_1 &= \overline{\Span} \{ g \ast B; ~g \in L^1(\mathbb C^n), ~B \in \mathcal L(F_t^p)\}\\
&= \{ B \in \mathcal L(F_t^p); ~f_s \ast B \to B \text{ in operator norm as } s \to 0\}.
\end{align*}
Lemma \ref{Lemma_continuity} also yields that $\mathcal R_t \ast \BUC(\mathbb C^n) \subseteq \mathcal C_1$.  It remains to show that $\mathcal R_t \ast \BUC(\mathbb C^n)$ is dense in $\mathcal C_1$. Let $B \in \mathcal C_1$. Then, we can choose $g \in L^1(\mathbb C^n)$ such that $\| B - g \ast B\|_{op} < \varepsilon$ according to Lemma \ref{Proposition_approximate_identity}. Since $\mathcal R_t\ast \mathcal N(F_t^p)$ is dense in $L^1(\mathbb C^n)$ by Proposition \ref{proposition_density}, we can choose $C \in \mathcal N(F_t^p)$ such that $\| g - \mathcal R_t \ast C \|_{L^1} < \varepsilon$. Combining this, we obtain
\[ \| B - \mathcal R_t \ast (C \ast B)\|_{op} \leq \| B - g \ast B\|_{op} + \| (g - \mathcal R_t \ast C) \ast B\|_{op} \leq (1 + \| B\|_{op})\varepsilon, \]
and $C \ast B \in \BUC(\mathbb C^n)$ due to Lemma \ref{Lemma_continuity}. Observe that the associativity used here is not an issue by Lemma \ref{lemma_convolution_ainfty}. Therefore, $\mathcal R_t \ast \BUC(\mathbb C^n)$ is dense in $\mathcal C_1$. Hence, we have proven
\begin{align*}
\mathcal C_1 &= \overline{\Span} \{ g \ast B; ~g \in L^1(\mathbb C^n), ~B \in \mathcal L(F_t^p)\}\\
&= \{ B \in \mathcal L(F_t^p); ~ g_s \ast B \to B \text{ in operator norm as } s \to 0 \}\\
&= \overline{\mathcal R_t \ast \BUC(\mathbb C^n)}.
\end{align*}
As already mentioned, the equality
\[ \mathcal C_1 = \{ g \ast B; ~ g \in L^1(\mathbb C^n), ~B \in \mathcal C_1\} \]
is now a consequence of the Cohen-Hewitt factorization theorem.
\end{proof}
\subsection{Correspondence Theory}
This section closely follows the setup introduced by R. Werner in \cite{Werner1984}. Definitions and proofs are as presented in that paper, but we added a certain amount of details for convenience.
\begin{defn}
A subspace $\mathcal D = \mathcal D_0 \oplus \mathcal D_1 \subseteq \mathcal A^{\infty}$ is said to be a \emph{pair} if $\mathcal N(F_t^p) \ast \mathcal D \subseteq \mathcal D$. In this case, $\mathcal D_0$ and $\mathcal D_1$ are called \emph{corresponding spaces}.
\end{defn}
Before we turn to the main result of Correspondence Theory below, we will need the following lemma:
\begin{lem}\label{Lemma_translation_invariant}
Let $\mathcal D \subseteq \mathcal C$ be closed and $\alpha$-invariant. Then, we have
\[ L^1(\mathbb C^n) \ast \mathcal D \subseteq \mathcal D. \]
\end{lem}
\begin{proof}
As above, for $(f \oplus A) \in \mathcal D$ and $g \in L^1(\mathbb C^n)$ the convolution $g \ast (f \oplus A)$ is defined as a Bochner integral with integrand $z \mapsto g(z) \alpha_z(f \oplus A)$ taking values in the Banach space $\mathcal D$, hence the integral is naturally contained in $\mathcal D$.
\end{proof}
The following result is \cite[Theorem 4.1]{Werner1984}.
\begin{thm}\label{Theorem_correspondence_theory}
\begin{enumerate}[(1)]
\item If $\mathcal D$ is a pair, then $\overline{\mathcal D}$ is also a pair.
\item Let $\mathcal D$ be a pair. Then, $\mathcal R_t \ast \mathcal D_0$ is $\| \cdot\|_{op}$-dense in $\mathcal D_1 \cap \mathcal C_1$ and $P_{\mathbb C} \ast \mathcal D_1$ is $\| \cdot \|_{L^\infty}$-dense in $\mathcal D_0 \cap \mathcal C_0$.
\item Let $\mathcal D$ be a pair. Then,
\begin{align*}
A \in \mathcal C_1 \text{ and } P_{\mathbb C} \ast A \in \mathcal D_0 &\Longrightarrow A \in \overline{\mathcal D_1}\\
f \in \mathcal C_0 \text{ and } \mathcal R_t \ast f \in \mathcal D_1 &\Longrightarrow f \in \overline{\mathcal D_0},
\end{align*}
where closures are taken with respect to $\| \cdot \|_{op}$ and $\| \cdot \|_{L^\infty}$, respectively.
\item For each closed $\alpha$-invariant subspace $\mathcal D_0 \subseteq \mathcal C_0$ there is a unique closed and $\alpha$-invariant corresponding subspace $\mathcal D_1 \subseteq \mathcal C_1$ and vice versa.
\end{enumerate}
\noindent For the unique correspondences of closed, $\alpha$-invariant subspaces in part (4) of the theorem we write $\mathcal D_0 \leftrightarrow \mathcal D_1$.
\end{thm}
Although the theorem was initially formulated only for the Hilbert space case, its proof carries over to our setting of $p$-Fock spaces unchanged. We present the initial proof from \cite{Werner1984}:
\begin{proof}
\begin{enumerate}[(1)]
\item Let $f \in \overline{\mathcal D_0}$, $f_j \in \mathcal D_0$ such that $f_j \to f$ and $A \in \mathcal N(F_t^p)$. Since $\mathcal D$ is a pair, $A \ast f_j \in \mathcal D_1$ for all $j$. By Lemma \ref{lemma_convolution_ainfty},
\[ \|A \ast (f - f_j)\|_{op} \leq \| A\|_{\mathcal N} \| f - f_j\|_{L^\infty} \to 0, \quad j \to \infty, \]
hence $A \ast f\in \overline{\mathcal D_1}$. $\mathcal N(F_t^p) \ast \overline{\mathcal D_1} \subset \overline{\mathcal D_0}$ follows analogously.
\item $\mathcal R_t \ast \mathcal D_0 \subseteq \mathcal D_1 \cap \mathcal C_1$ follows since $\mathcal D$ is a pair and by Lemma \ref{Lemma_continuity}. Let $A \in \mathcal D_1\cap \mathcal C_1$ be arbitrary and $\varepsilon > 0$. By Lemma \ref{Proposition_approximate_identity} there exists $h \in L^1(\mathbb C^n)$ such that $\| A - h \ast A\|_{op} < \varepsilon$. Further, by Proposition \ref{proposition_density} there exists $B \in \mathcal N(F_t^p)$ such that $\| h - \mathcal R_t \ast B \|_{L^1} < \varepsilon.$ Together, we obtain
\[ \| A - (\mathcal R_t \ast B) \ast A\|_{op} < \varepsilon(1+\| A\|_{op}). \]
Finally, since $\mathcal N(F_t^p) \ast \mathcal D_1 \subseteq \mathcal D_0$, it holds $B \ast A \in \mathcal D_0$. Density of $P_{\mathbb C} \ast \mathcal D_1$ in $\mathcal D_0$ follows analogously.
\item The reasoning of (2) with the assumption $P_{\mathbb C} \ast A \in \mathcal D_0$ instead of $A \in \mathcal D_1$ proves the first implication, the second follows analogously.
\item Assume $\mathcal D_0 \subseteq \BUC(\mathbb C^n)$ is closed and $\alpha$-invariant. Define the spaces
\begin{align*}
\mathcal D_1^{-} &:= \mathcal D_0 \ast \mathcal N(F_t^p) \subseteq \mathcal C_1 \\
\mathcal D_1^{+} &:= \{ A \in \mathcal C_1; ~A \ast B \in \mathcal D_0 \text{ for each } B \in \mathcal N(F_t^p) \}.
\end{align*}
By Lemma \ref{Lemma_translation_invariant} we have $L^1(\mathbb C^n) \ast \mathcal D_0 \subseteq \mathcal D_0$. $\mathcal D_0 \oplus \mathcal D_1^{-}$ is a pair since
\[ \mathcal N(F_t^p) \ast \mathcal D_1^{-}  = \mathcal N(F_t^p) \ast \mathcal N(F_t^p) \ast \mathcal D_0 \subseteq L^1(\mathbb C^n) \ast \mathcal D_0 \subset \mathcal D_0,\]
by Lemma \ref{Lemma_translation_invariant}. Further, by definition we have $\mathcal N(F_t^p) \ast \mathcal D_1^{+} \subseteq \mathcal D_0$. One easily checks $\mathcal D_1^{-} \subseteq \mathcal D_1^{+}$, hence $\mathcal D_0 \oplus \mathcal D_1^{+}$ is also a pair. If $\mathcal E_1 \subseteq \mathcal L(F_t^p)$ is any other subspace such that $\mathcal D_0 \oplus \mathcal E_1$ is a pair, then
\[ \mathcal D_1^{-} \subseteq \mathcal E_1 \subseteq \mathcal D_1^{+}. \]
Let $A \in \mathcal D_1^{+}$. Then, part (3) of the theorem applied to the pair $\mathcal D_0 \oplus \mathcal D_1^{-}$ yields $A \in \overline{D_1^{-}}$, which proves the result.

The other direction of the correspondence is proven analogously.
\end{enumerate}
\end{proof}
We want to note that there is an analogous correspondence theory for pairs $\mathcal D = \mathcal D_0 \oplus \mathcal D_1 \subseteq \mathcal A^{p_0}$, $1 \leq p_p < \infty$, with identical statements and proofs (up to the obvious changes). Further, Werner formulated his version of the above theorem using the notion of \emph{regular operators}, which we did not introduce. Our operators $P_{\mathbb C}$ and $\mathcal R_t$ are such regular operators. We will not dwell on this and refer to Werner's original work \cite{Werner1984}.

\section{Applications to Toeplitz algebras}\label{section_toeplitz_operators}
Let $S \subseteq L^\infty(\mathbb C^n)$. Then, we denote by $\mathcal T^{p,t}(S) \subset \mathcal L(F_t^p)$ the Banach algebra generated by all Toeplitz operators with symbols in $S$. Let us denote by $\mathcal T^{p,t} := \mathcal T^{p,t}(L^\infty(\mathbb C^n))$ the full Toeplitz algebra. By $\mathcal T_{lin}^{p,t}(S) \subset \mathcal L(F_t^p)$ we denote the closed linear span of Toeplitz operators with symbols in $S \subset L^\infty(\mathbb C^n)$. Finally, in the case of $p = 2$ we will use $\mathcal T_{\ast}^{2,t}(S)\subset \mathcal L(F_t^2)$ for the $C^\ast$ algebra generated by Toeplitz operators with symbols in $S$.
The following result is an immediate consequence of Proposition \ref{proposition_characterizations_C_1}:
\begin{thm}\label{Toeplitz_algebra_linearly_generated}
We have
\begin{align*}
\mathcal T^{p,t} &= \mathcal T_{lin}^{p,t}(\BUC(\mathbb C^n)) \\
&=  \{ A \in \mathcal L(F_t^p); ~ z \mapsto W_z A W_{-z} \text{ is continuous w.r.t. the operator norm}\}\\
&= \{ A \in \mathcal L(F_t^p); ~ f_s \ast A \to A \text{ in operator norm as } s \to 0 \}\\
&= \{ g \ast A; ~ g \in L^1(\mathbb C^n), ~ A \in \mathcal C_1\}.
\end{align*}
\end{thm}
\begin{proof}
Using $\mathcal R_t \ast f = T_f^t$ for every $f \in L^\infty(\mathbb C^n)$, Proposition \ref{proposition_characterizations_C_1} gives 
\begin{equation*}
\mathcal C_1 = \overline{\mathcal R_t \ast \BUC(\mathbb C^n)} = \mathcal T_{lin}^{p,t}(\BUC(\mathbb C^n)).
\end{equation*}
Since every convolution between an operator from $\mathcal N(F_t^p)$ and a function from $L^\infty(\mathbb C^n)$ is in $\mathcal C_1$ by Lemma \ref{Lemma_continuity}, we obtain $T_f^t = \mathcal R_t \ast f \in \mathcal C_1$ for every $f \in L^\infty(\mathbb C^n)$. Using that $\mathcal C_1$ is a Banach algebra, we therefore get $\mathcal T^{p,t} \subseteq \mathcal C_1$. Finally, the inclusion $\mathcal T_{lin}^{p,t}(\operatorname{BUC}(\mathbb C^n)) \subseteq \mathcal T^{p,t}$ is trivial.
\end{proof}
\begin{rem}
The result $\mathcal T^{2,1} = \mathcal T_{lin}^{2,1}(L^\infty(\mathbb C^n))$ was obtained by J. Xia in \cite{Xia}. While his methods works for any $t > 0$, the assumption $p = 2$ was crucial in an important step of his proof. Hence, the above theorem improves that result.
\end{rem}
Theorem \ref{Toeplitz_algebra_linearly_generated} tells us: For each $A \in \mathcal T^{p,t}$ there is a sequence of Toeplitz operators $T_{f_j}^t$ such that $T_{f_j}^t \to A$ in operator norm. So far, we have no information how the symbols $f_j$ are related to the operator $A$. A careful investigation of the underlying theory can give us the answer. Recall that convolution by 
\[ f_s(z) = \frac{1}{(\pi s)^n} e^{-\frac{|z|^2}{s}} \]
is an approximate identity in $\mathcal C$. Let $N \in \mathbb N$. Since the span of functions of the form 
\[ \alpha_z(f_t) = \alpha_z(\mathcal R_t \ast P_{\mathbb C}) = \mathcal R_t \ast (\alpha_z(P_{\mathbb C})) \]
is dense in $L^1(\mathbb C^n)$ by Wiener's Theorem \cite[Theorem 9.5]{Rudin1991}, it follows that there are constants $c_j^N$ and points $z_j^N$ such that
\begin{equation}\label{equation_approximation}
\left \| f_{t/N} - \mathcal R_t \ast \left(\sum_{j=1}^{M_N} c_j^N \alpha_{z_j^N}(P_{\mathbb C}) \right) \right \|_{L^1} \leq \frac{1}{N}.
\end{equation}
Then, by the usual norm estimates for convolutions,
\begin{align*}
\left \| A - T_{\sum_{j=1}^{M_N} c_j^N \alpha_{z_j^N}(\widetilde{A})}^t \right \|_{op} &=\left \| A - \mathcal R_t \ast \left( \sum_{j=1}^{M_N} c_j^N \alpha_{z_j^N}(\widetilde{A}) \right) \right \|_{op}\\
&= \left \| A - \mathcal R_t \ast \left( \sum_{j=1}^{M_N} c_j^N \alpha_{z_j^N}(P_\mathbb{C}) \right) \ast A\right \|_{op}\\
&\leq \left \| A - f_{t/N} \ast A \right \|_{op}\\
&\quad  + \left \| f_{t/N} \ast A - \mathcal R_t \ast \left( \sum_{j=1}^{M_N} c_j^N \alpha_{z_j^N}(P_\mathbb{C}) \right) \ast A \right \|_{op}\\
&\leq \left \| A - f_{t/N} \ast A \right \|_{op} + C\| A\|_{op} \frac{1}{N}\\
&\to 0, \quad N \to \infty.
\end{align*}
Hence, we can approximate the operator $A$ by a sequence of Toeplitz operators, where each Toeplitz operator has a weighted version of the Berezin transform of $A$ as its symbol.

The following version of Werner's Correspondence Theorem \ref{Theorem_correspondence_theory} in the Toeplitz operator setting gives a different perspective on Theorem \ref{Toeplitz_algebra_linearly_generated}.
\begin{prop}\label{main_result_correspondence}
Let $\mathcal D_0 \subseteq \BUC(\mathbb C^n)$ be a closed $\alpha$-invariant subspace. Then, we have 
\[ \mathcal D_0 \leftrightarrow \mathcal T_{lin}^{p,t}(\mathcal D_0). \]
\end{prop}
\begin{proof}
By Proposition \ref{Proposition_Toeplitz_identity}, $\mathcal R_t \ast f= T_f^t$ for $f \in \mathcal D_0$. By Theorem \ref{Theorem_correspondence_theory}(2), the set of these operators is dense in the closed and $\alpha$-invariant subspace of $\mathcal C_1$ corresponding to $\mathcal D_0$.
\end{proof}
\begin{rem}
It is important to note and easy to prove that $\mathcal D_0$ is invariant under the operator $U$ if and only if $\mathcal T_{lin}^{p,t}(\mathcal D_0)$ is invariant under adjoining $U$ from left and right, i.e. for $A \in \mathcal T_{lin}^{p,t}(\mathcal D_0)$ we have $UAU \in \mathcal T_{lin}^{p,t}(\mathcal D_0)$. Further, the correspondence is ordered by inclusion: If $\mathcal E_0 \subseteq \BUC(\mathbb C^n)$ is also closed and $\alpha$-invariant, then $\mathcal D_0 \subseteq \mathcal E_0$ if and only if $\mathcal T_{lin}^{p,t}(\mathcal D_0) \subseteq \mathcal T_{lin}^{p,t}(\mathcal E_0)$.
\end{rem}
As a simple application of this correspondence, we obtain the following result. Part (2) is well know, see e.g. \cite{Berger_Coburn1994} for the Hilbert space case. Part (3) was the essential result in \cite{Bauer_Isralowitz2012}, yet our proof is now significantly shorter. Observe that part (1) was also mentioned in \cite{Werner1984} in the Schr\"odinger representation.
\begin{prop}\label{proposition_classification_compact}\begin{enumerate}[(1)]
\item We have
\begin{align*}
C_0(\mathbb C^n) \leftrightarrow \mathcal K(F_t^p).
\end{align*}
\item The full ideal of compact operators is generated by Toeplitz operators, i.e. 
\[ \mathcal K(F_t^p) = \overline{\{T_f^t; ~f \in C_0(\mathbb C^n)\}}^{\| \cdot\|_{op}}. \]
\item An operator $A \in \mathcal L(F_t^p)$ is compact if and only if $A \in \mathcal T^{p,t}$ and $\widetilde{A} \in C_0(\mathbb C^n)$.
\end{enumerate}
\end{prop}
\begin{proof}
Since $C_0(\mathbb C^n)$ is closed and translation invariant, we have
\[ C_0(\mathbb C^n) \leftrightarrow \mathcal T_{lin}(C_0(\mathbb C^n)) \subseteq \mathcal K(F_t^p). \]
Further, $\mathcal K(F_t^p) \subset \mathcal T^{p,t}$ is also closed and $\alpha$-invariant. For $A \in \mathcal K(F_t^p)$ it is well-known that $ \widetilde{A} \in C_0(\mathbb C^n)$, hence
\[ \mathcal K(F_t^p) \leftrightarrow \mathcal D_0 \subseteq C_0(\mathbb C^n) \]
which proves the correspondence. The second result is a direct consequence of Proposition \ref{main_result_correspondence}. The third result now follows from this correspondence: If $A \in \mathcal T^{p,t} = \mathcal C_1$ and $\widetilde{A} = P_{\mathbb C} \ast A \in C_0(\mathbb C^n)$, then Theorem \ref{Theorem_correspondence_theory}(3) applied to the pair $C_0(\mathbb C^n) \oplus \mathcal K(F_t^p)$ yields $A \in \mathcal K(F_t^p)$. The other direction (i.e. $A \in \mathcal K(F_t^p) \Rightarrow A \in \mathcal T^{p,t}$ and $\widetilde{A} \in C_0(\mathbb C^n)$) is immediate.
\end{proof}

\subsection{Linearly generated Toeplitz algebras}
If one is thinking about possible generalizations of Theorem \ref{Toeplitz_algebra_linearly_generated}, one could ask if a given Banach algebra of Toeplitz operators is linearly generated by some set of functions. If the algebra is $\alpha$-invariant, the answer to this is positive. The following statement is in a sense dual to Proposition \ref{main_result_correspondence}.
\begin{thm}\label{theorem_linearly_generated}
Let $\mathcal S \subseteq \mathcal T^{p,t}$ be an $\alpha$-invariant closed subspace (not necessarily an algebra). Then, there is some closed and $\alpha$-invariant subspace $\mathcal D_0 \subseteq \BUC(\mathbb C^n)$ such that
\[ \mathcal S = \mathcal T_{lin}^{p,t}(\mathcal D_0). \]
Further, $\mathcal D_0$ is given by $\closedspan\{ \widetilde{A}; ~ A \in \mathcal S\}$.
\end{thm}
\begin{proof}
Since $\mathcal S \subseteq \mathcal T^{p,t}$, $\alpha$ acts norm-continuously on $\mathcal S$. Theorem \ref{Theorem_correspondence_theory} asserts that there is some closed and $\alpha$-invariant $\mathcal D_0 \subseteq \BUC(\mathbb C^n)$ such that
\[ \mathcal S \leftrightarrow \mathcal D_0 = \closedspan\{ \widetilde{A}; ~ A \in \mathcal S\}.\]
As we have seen before, 
\[ \mathcal D_0 \leftrightarrow \mathcal T_{lin}^{p,t}(\mathcal D_0). \]
\end{proof}
Before we continue with the general theory, want to present two examples of linearly generated Toeplitz algebras:
\begin{exs}
\begin{enumerate}[(1)]
\item \textbf{The CCR algebra.}
In \cite{Coburn1999}, L. Coburn studied the CCR algebra in the Bargmann representation. For this, let $p = 2, t = 1/2$ and define
\begin{equation*}
\operatorname{CCR}(\mathbb C^n) := C^\ast(\{ W_z; ~z \in \mathbb C^n\}).
\end{equation*}
Here, $C^\ast(\mathcal A)$ denotes the $C^\ast$ algebra generated by the set $\mathcal A \subseteq \mathcal L(F_{1/2}^2)$. It is well-known that each Weyl operator $W_z$ coincides with a certain Toeplitz operator, hence $\operatorname{CCR}(\mathbb C^n) \subseteq \mathcal T^{2,1/2}$. Coburn showed
\[ \operatorname{CCR}(\mathbb C^n) = \mathcal T^{2,1/2}(\operatorname{AP}) = \mathcal T_{lin}^{2,1/2}(\operatorname{AP}), \]
where $\operatorname{AP}$ denotes the Banach algebra of almost periodic functions. Note that $\operatorname{AP}$ can be seen to be $\alpha$-invariant. Hence, certain $\alpha$-invariant closed symbol spaces may generate the same space of operators linearly and as an algebra.
\item \textbf{Radial Toeplitz operators.}
In the study of commutative $C^\ast$ algebras generated by Toeplitz operators on the Fock spaces $F_t^2$, there are two relevant model cases: The radial case and the horizontal case. The $C^\ast$ algebra generated by Toeplitz operators with radial symbols on $F_t^2$ is linearly generated by radial Toeplitz operators. Yet, the space of radial functions is not $\alpha$-invariant. Hence, there are interesting aspects of the problem of linearly generated Toeplitz algebras outside the scope of our approach.
\end{enumerate}
\end{exs} \ \\
$\operatorname{CCR}(\mathbb C^n)$ is an example of a Toeplitz algebra which is generated as an algebra and as a closed linear space by the same set $\mathcal D_0 \subseteq \BUC(\mathbb C^n)$. We will further investigate such subspaces of $\mathcal T^{p,t}$. 
\begin{lem}
If $\mathcal D_0 \subseteq \BUC(\mathbb C^n)$ is closed and $\alpha$-invariant, then $\mathcal D_0$ is closed under the operation $f \mapsto \widetilde{f}^{(t)}$ for all $t > 0$.
\end{lem}
\begin{proof}
For $f \in \mathcal D_0$ we have $T_f^t = \mathcal R_t \ast f \in \mathcal D_1$ and $P_{\mathbb C} \ast (\mathcal R_t \ast f) = P_{\mathbb C} \ast T_f^t = \widetilde{f}^{(t)} \in \mathcal D_0$, where $\mathcal D_0 \leftrightarrow \mathcal D_1$.
\end{proof}
We obtain the following general criterion for a Toeplitz algebra being linearly generated by the same class of symbols:
\begin{thm}
Let $\mathcal D_0 \subseteq \BUC(\mathbb C^n)$ be closed and $\alpha$-invariant. Then, we have
\[ \mathcal T_{lin}^{p,t}(\mathcal D_0) = \mathcal T^{p,t}(\mathcal D_0) \]
if and only if for each $k \in \mathbb N$ the range of the map
\begin{equation}\label{product_map}
\mathcal D_0^k \mapsto \BUC(\mathbb C^n), \quad (f_1, \dots, f_k) \mapsto (T_{f_1}^t \dots T_{f_k}^t)^{\sim}
\end{equation}
is contained in $\mathcal D_0$. Further, we have
\[ \mathcal T_{lin}^{2,t}(\mathcal D_0) = \mathcal T_{\ast}^{2,t}(\mathcal D_0) \]
if and only if $\mathcal D_0$ is closed under the product maps (\ref{product_map}) and under taking complex conjugates.
\end{thm}
\begin{proof}
Assume $\mathcal T_{lin}^{p,t}(\mathcal D_0) = \mathcal T^{p,t}(\mathcal D_0)$. Then, the operator product $T_{f_1}^t \dots T_{f_k}^t$ is contained in $\mathcal T_{lin}^{p,t}(\mathcal D_0)$. The Berezin transform of that operator product, which is just convolution by $P_{\mathbb C}$, is therefore contained in $\mathcal D_0$.

Now, assume that the range of the map (\ref{product_map}) is contained in $\mathcal D_0$. We need to prove that $T_{f_1}^t\dots T_{f_k}^t \in \mathcal T_{lin}^{p,t}(\mathcal D_0)$ for $f_1, \dots, f_k \in \mathcal D_0$. Since $T_{f_1}^t \dots T_{f_k}^t \in \mathcal C_1$ and $(T_{f_1}^t\dots T_{f_k}^t)^{\sim} = P_{\mathbb C} \ast (T_{f_1}^t \dots T_{f_k}^t) \in \mathcal D_0$ by assumption, this follows from Theorem \ref{Theorem_correspondence_theory}(3).

Finally, let $p = 2$ and let $\mathcal D_0$ be closed under the product maps (\ref{product_map}). If $\mathcal D_0$ is further closed under taking complex conjugates, then the adjoint of each generator of $\mathcal T^{2,t}(\mathcal D_0)$ is also contained in $\mathcal T^{2,t}(\mathcal D_0)$, hence it is a $C^\ast$ algebra. If on the other hand $\mathcal T_{lin}^{2,t}(\mathcal D_0) = \mathcal T_{\ast}^{2,t}(\mathcal D_0)$, then for each $f \in \mathcal D_0$ it holds 
\[ (T_f^t)^\ast = T_{\overline{f}}^t = \mathcal R_t \ast \overline{f} \in \mathcal T_{lin}^{2,t}(\mathcal D_0). \] By Theorem \ref{Theorem_correspondence_theory}(3) we therefore obtain $\overline{f} \in \mathcal D_0$.
\end{proof}
While the previous result gives a characterization of all $\mathcal D_0$ which have the desired property, it seems that the property of being closed under the above product maps is in general difficult to verify. Yet, there is an important consequence: Since the Toeplitz operators are integral operators and the integral expression defining them does not depend on $p$, and also the formula for the Berezin transform does not depend on $p$, the property that the range of (\ref{product_map}) is contained in $\mathcal D_0$ does not depend on $p$. Hence, we obtain:
\begin{cor}\label{p_independence}
Let $\mathcal D_0 \subseteq \BUC(\mathbb C^n)$ be closed and $\alpha$-invariant. Then, we have
\[ \mathcal T_{lin}^{p,t}(\mathcal D_0) = \mathcal T^{p,t}(\mathcal D_0) \]
for one $p$ if and only if it holds true for all $p$.
\end{cor}
Analogously to the above reasoning one proves:
\begin{cor}\label{corollary_ideal}
Let $\mathcal D_0 \subseteq \BUC(\mathbb C^n)$ be closed and $\alpha$-invariant such that
\[ \mathcal T_{lin}^{p,t}(\mathcal D_0) = \mathcal T^{p,t}(\mathcal D_0). \]
If $\mathcal I \subset \mathcal D_0$ is also closed and $\alpha$-invariant, then $\mathcal T_{lin}^{p,t}(\mathcal I)$ is a (left/right/two-sided) ideal in $\mathcal T_{lin}^{p,t}(\mathcal D_0)$ for one $p$ if and only if it is a (left/right/two-sided) ideal for all $p$.
\end{cor}
While there is no $C^\ast$ algebraic structure on the operator side (at least for $p \neq 2$), the fact that $\mathcal T^{p,t}$ depends ``almost not on $p$'' (without making this a precise statement) still allows us to obtain results typical for $C^\ast$ algebras. Here is one example:
\begin{prop}\label{Proposition_contains_compacts}
Let $\mathcal D_0 \subseteq \BUC(\mathbb C^n)$ be an $\alpha$-invariant $C^\ast$ subalgebra. If $\mathcal T_{lin}^{p,t}(\mathcal D_0)$ contains at least one non-trivial compact operator, then it contains all compact operators.
\end{prop}
The proof of this proposition is based on the following well-known fact, which is in turn a simple exercise using the Stone-Weierstrass Theorem.
\begin{lem}
Let $\mathcal D_0$ be an $\alpha$-invariant $C^\ast$ subalgebra of $\BUC(\mathbb C^n)$. If $\mathcal D_0$ contains a non-trivial function from $C_0(\mathbb C^n)$, then it holds $C_0(\mathbb C^n) \subseteq \mathcal D_0$.
\end{lem}
\begin{proof}[Proof of Proposition \ref{Proposition_contains_compacts}]
Let $0 \neq K \in \mathcal T_{lin}^{p,t}(\mathcal D_0)$ be compact. Then, $0 \neq P_{\mathbb C} \ast K \in C_0(\mathbb C^n) \cap \mathcal D_0$. The previous lemma implies $C_0(\mathbb C^n) \subseteq \mathcal D_0$, therefore $\mathcal R_t \ast C_0(\mathbb C^n) \subseteq \mathcal T_{lin}^{p,t}(\mathcal D_0)$. Since $\overline{\mathcal R_t \ast C_0(\mathbb C^n)} = \mathcal K(F_t^p)$ (Proposition \ref{proposition_classification_compact}), this yields $\mathcal K(F_t^p) \subseteq \mathcal T_{lin}^{p,t}(\mathcal D_0)$.
\end{proof}

The remaining part of this paper will deal with a proof of our second main theorem, which is the following:
\begin{thm}\label{main_theorem}
Let $\mathcal D_0 \subseteq \BUC(\mathbb C^n)$ be closed, $\alpha$- and $U$-invariant. Then, the following are equivalent:
\begin{enumerate}[(i)]
\item $\mathcal D_0$ is a $C^\ast$ algebra with respect to the standard operations and $L^\infty(\mathbb C^n)$ norm;
\item $\mathcal T_{lin}^{2,t}(\mathcal D_0) = \mathcal T_{\ast}^{2,t}(\mathcal D_0)$ for all $t > 0$.
\end{enumerate}
If the above equivalent conditions are fulfilled, then we have $\mathcal T_{lin}^{p,t}(\mathcal D_0) = \mathcal T^{p,t}(\mathcal D_0)$ for all $1 < p < \infty, ~ t > 0$.

If $\mathcal D_0$ is a closed, $\alpha$- and $U$-invariant $C^\ast$ subalgebra of $\BUC(\mathbb C^n)$ and $\mathcal I \subset \mathcal D_0$ is closed and $\alpha$-invariant, then the following are equivalent:
\begin{enumerate}[(i*)]
\item $\mathcal I$ is an ideal in $\mathcal D_0$;
\item $\mathcal T_{lin}^{2,t}(\mathcal I)$ is a left or right ideal in $\mathcal T_{lin}^{2,t}(\mathcal D_0)$ for all $t > 0$;
\item $\mathcal T_{lin}^{2,t}(\mathcal I)$ is a two-sided ideal in $\mathcal T_{lin}^{2,t}(\mathcal D_0)$ for all $t > 0$.
\end{enumerate}
Under these assumptions, $\mathcal T_{lin}^{p,t}(\mathcal I) = \mathcal T^{p,t}(\mathcal I)$ is a closed and two-sided ideal in $\mathcal T_{lin}^{p,t}(\mathcal D_0)$ for all $1 < p < \infty$ and $t > 0$.
\end{thm}
In this theorem, recall that $U \in \mathcal L(F_t^p)$ is the operator $Uf(z) = f(-z)$.
\subsection{Applying quantization estimates}
Certain quantization estimates for the Toeplitz quantization on Fock spaces have been studied intensively in recent years. We apply these estimates now for proving the first part of our theorem.
\begin{prop}\label{proposition_quantization_results}
\begin{enumerate}[1)]
\item Let $\mathcal D_0 \subseteq \BUC(\mathbb C^n)$ be a closed and $\alpha$-invariant subspace such that for all $t > 0$ it holds
\begin{align*}
\mathcal T_{lin}^{p,t} (\mathcal D_0) = \mathcal T^{p,t}(\mathcal D_0).
\end{align*}
Then, $\mathcal D_0$ is a Banach algebra under pointwise multiplication. If it even holds
\begin{align*}
\mathcal T_{lin}^{2,t}(\mathcal D_0) = \mathcal T_{\ast}^{2,t}(\mathcal D_0)
\end{align*}
for all $t > 0$, then $\mathcal D_0$ is even a $C^\ast$ algebra.
\item Let $\mathcal D_0$ be such that it fulfils all assumptions from part 1). Further, let $\mathcal I \subset \mathcal D_0$ be closed and $\alpha$-invariant such that $\mathcal T_{lin}^{p,t}(\mathcal I)$ is either a left- or a right-ideal for all $t > 0$. Then, $\mathcal I$ is an ideal in $\mathcal D_0$.
\end{enumerate}
\end{prop}
\begin{proof}
\begin{enumerate}[1)]
\item By Corollary \ref{p_independence}, we may assume that $p = 2$. We need to prove that the pointwise product of two elements is contained in $\mathcal D_0$. Therefore, let $f, g \in \mathcal D_0$. For each $t > 0$, $\mathcal D_0$ is closed under the product
\[ (f,g) \mapsto \widetilde{T_f^t T_g^t}. \]
Since $f, g \in \BUC(\mathbb C^n)$, it holds \cite{Bauer_Coburn_Hagger}
\[ \| T_f^t T_g^t - T_{fg}^t\|_{op} \to 0, \quad t \to 0\]
which implies
\[ \| \widetilde{T_f^t T_g^t} - \widetilde{T_{fg}^t} \|_{L^\infty} \to 0, \quad t \to 0. \]
Since $fg \in \BUC(\mathbb C^n)$, we obtain \cite{Bauer_Coburn2015}
\[ \| fg - \widetilde{fg}^{(t)}\|_{L^\infty} \to 0, \quad t \to 0. \]
Finally,
\begin{align*}
\| \widetilde{T_f^t T_g^t} - fg\|_{L^\infty} &\leq \| \widetilde{T_f^t T_g^t} - \widetilde{T_{fg}^t}\|_{L^\infty} + \| \widetilde{fg}^{(t)} - fg\|_{L^\infty}\\
&\to 0, \quad t \to 0.
\end{align*}
Since $\widetilde{T_f^t T_g^t} \in \mathcal D_0$ for each $t > 0$, and $\mathcal D_0$ is assumed to be norm-closed, the result follows.
\item Follows from analogous reasoning as part 1), assuming $f \in \mathcal D_0$ and $g \in \mathcal I$ or vice versa.
\end{enumerate}
\end{proof}
This proposition of course implies $(ii) \Rightarrow (i)$ and $(ii^\ast) \Rightarrow (i^\ast)$ of Theorem \ref{main_theorem}.

\subsection{Drawing information from limit operators}
In what follows, we will denote by a \textit{non-separating compactification} of $\mathbb C^n$ a pair $(\psi, X)$, where $X$ is a compact topological space $X$ and a continuous map $\psi: \mathbb C^n \to X$ such that $\psi(\mathbb C^n)$ is dense in $X$. Observe that, in contrast to the standard definition of a compactification, we do not assume $\psi$ to be injective - this is why we use the notion of non-separating compactifications. If $(\psi, X)$ and $(\varphi, Y)$ are two non-separating compactifications of $\mathbb C^n$, we say that $(\psi, X)$ and $(\varphi, Y)$ are equivalent (and write $(\psi, X) \sim (\varphi, Y)$) if there is a homeomorphism $\theta: X \to Y$ such that $\theta \circ \psi = \varphi$.

In this sense, there is a 1-1 correspondence between unital $C^\ast$ subalgebras of $C_b(\mathbb C^n)$ and equivalence classes of non-separating compactifications: Given a unital $C^\ast$  subalgebra $\mathcal A$ of $C_b(\mathbb C^n)$, a representative of the equivalence class of non-separating compactifications is given by $(\text{ev}, \mathcal M(\mathcal A))$, where $\mathcal M(\mathcal A)$ is the maximal ideal space of $\mathcal A$ with $w^\ast$ topology, and $\text{ev}$ is the map
\[ \text{ev}: \mathbb C^n \to \mathcal M(\mathcal A), \quad \text{ev}(x)(f) = f(x). \] 
On the other hand, given any non-separating compactification $(\psi, X)$, we define $\mathcal A_{(\psi, X)} \subseteq C_b(\mathbb C^n)$ through
\[ \mathcal A_{(\psi, X)} := \{ f \circ \psi; ~ f \in C(X) \}, \]
which is a unital $C^\ast$ subalgebra of $C_b(\mathbb C^n)$. One readily checks that for each unital $C^\ast$ subalgebra $\mathcal A$ of $C_b(\mathbb C^n)$ and each non-separating compactification $(\psi, X)$ of $\mathbb C^n$ we have
\begin{align} \label{algebra_equality}
\mathcal A_{(\text{ev}, \mathcal M (\mathcal A))} &= \mathcal A,\\
(\psi, X) &\sim (\text{ev}, \mathcal M(\mathcal A_{(\psi, X)})) \nonumber.
\end{align} 
Observe that this correspondence restricts to a 1-1 correspondence between compactifications in the usual sense (i.e. where the map $\psi$ is further assumed to be injective) and unital $C^\ast$ subalgebras of $C_b(\mathbb C^n)$ which separate points. 

In what follows, let $\mathcal A$ denote an $\alpha$- and $U$-invariant unital $C^\ast$ subalgebra of $\BUC(\mathbb C^n)$, i.e. for each $f \in \mathcal A$ and $z \in \mathbb C^n$ we have $\alpha_z(f) \in \mathcal A$ and $Uf \in \mathcal A$. Under these assumptions, one can show that $\mathcal M(\mathcal A)$ is actually a compactification (in the usual sense) of $\mathbb C^n/\Gamma$ for some closed additive subgroup $\Gamma$ of $\mathbb C^n$. While this is not important in the following, it serves as a geometric picture for what is going on.

We will always identify $\mathbb C^n$ with its image in the compactification corresponding to $\mathcal A$. Let $f \in \mathcal A$ and $z \in \mathbb C^n$. For $z, w \in \mathbb C^n$ we have
\begin{equation}\label{shifted_function}
\alpha_z(f)(w) = \text{ev}(z)(U\alpha_w(f)).
\end{equation}
For $x \in \mathcal M(\mathcal A)$ and $w \in \mathbb C^n$ we define
\[ f_x(w) = x(U\alpha_w(f)), \]
where we will abuse the notation $f_z = f_{\text{ev}(z)}$ for $z \in \mathbb C^n$, which is justified by Equation (\ref{shifted_function}). As in \cite{Bauer_Isralowitz2012}, one proves that $f_x \in \BUC(\mathbb C^n)$ for all $x \in \mathcal M(\mathcal A)$. Let $(z_\gamma)_\gamma \subset \mathbb C^n$ be any net converging to $x \in \mathcal M(\mathcal A)$. Then, $\alpha_{z_\gamma}(f) = f_{z_\gamma} \to f_x$ pointwise. An easy application of the Arzel\`{a}-Ascoli theorem shows that this convergence is even uniform on compact subsets. On the level of Toeplitz operators, one can then show that $\alpha_{z_\gamma}(T_f^t) = T_{\alpha_{z_\gamma}(f)}^t \overset{\gamma}{\to} T_{f_x}^t$, where the convergence is in strong operator topology \cite{Bauer_Isralowitz2012}. This has the following important consequence, which follows as in \cite{Bauer_Isralowitz2012} with only minor changes (since we are using a possibly smaller compactification of $\mathbb C^n$ than $\mathcal M(\BUC(\mathbb C^n))$):
\begin{prop}
For each $A \in \mathcal T_{lin}^{p,t}(\mathcal A)$, the map
\[ \mathbb C^n \to \mathcal T_{lin}^{p,t}(\mathcal A), \quad z \mapsto \alpha_z(A) \]
extends to a continuous map
\[ \mathcal M(\mathcal A) \to (\mathcal T_{lin}^{p,t}(\BUC(\mathbb C^n)), ~SOT^\ast), \quad x \mapsto A_x,\]
which is norm-bounded by $\| A\|_{op}$. Here, $SOT^\ast$ denotes the strong$^\ast$ operator topology, i.e. both functions $x \mapsto A_x$ and its Banach space adjoint (understood in the dual pairing coming from $F_t^2$) $x \mapsto A_x^\ast$ are continuous with respect to the strong operator topology.
\end{prop}
\begin{rem}
For each $A \in \mathcal T_{lin}^{p,t}(\mathcal A)$, the operators $A_x$ ($x \in \mathcal M(\mathcal A) \setminus \mathbb C^n$) are called the \emph{limit operators} of $A$. They are of independent interest, as they can be used to investigate the Fredholm property and essential spectrum of $A$ \cite{Fulsche_Hagger}.
\end{rem}
Let $A, B \in \mathcal T_{lin}^{p,t}(\mathcal A)$. Then, both $z \mapsto \alpha_z(A)$ and $z \mapsto \alpha_z(B)$ are uniformly bounded in norm by $\| A\|_{op}$ and $\| B\|_{op}$, respectively. Therefore,
\[ z \mapsto \alpha_z(AB) = \alpha_z(A) \alpha_z(B) \]
extends to a strongly continuous map 
\[ \mathcal M(\mathcal A) \ni x \mapsto (AB)_x = A_x B_x \in \mathcal T^{p,t}, \]
i.e. passing to the limit operators is multiplicative. Further, since $(\mathcal L(F_t^p))^\ast \ni W_z^\ast = W_{-z} \in \mathcal L(F_t^q)$ ($1/p + 1/q = 1$), one sees that $(A_x)^\ast = (A^\ast)_x$.
\begin{prop}\label{proposition_unital_algebra}
Let $\mathcal A$ be a unital, $\alpha$- and $U$-invariant C$^\ast$ subalgebra of $\BUC(\mathbb C^n)$. Then we have
\[ \mathcal T_{lin}^{p,t}(\mathcal A) = \mathcal T^{p,t}(\mathcal A) \]
and
\[ \mathcal T_{lin}^{2,t}(\mathcal A) = \mathcal T_{\ast}^{2,t}(\mathcal A) \]
for all $t > 0$ and $1 < p < \infty$.
\end{prop}
\begin{proof}
Observe that the second statement and the first statement are equivalent, since the first is $p$-independent by Corollary \ref{p_independence} and $\mathcal T^{2,t}(\mathcal A) = \mathcal T_{\ast}^{2,t}(\mathcal A)$ follows since $\mathcal A$ is closed under taking complex conjugates. Therefore, it suffices to prove that $\mathcal T_{lin}^{2,t}(\mathcal A)$ is closed under taking products.

Let $A, B \in \mathcal T_{lin}^{2,t}(\mathcal A)$. As noted above,
\[ z \mapsto \alpha_z(AB) \]
extends to a strongly continuous map
\[ \mathcal M(\mathcal A) \ni x \mapsto (AB)_x \in \mathcal T^{2,t}. \]
For the Berezin transform of the product, we have
\[ U(AB)^\sim(z) = (AB)^\sim(-z) = (\alpha_z(AB))^{\sim}(0) = \langle (AB)_z1, 1\rangle_{F_t^2}, \]
which extends by strong continuity of $x \mapsto (AB)_x$ to $U(AB)^\sim \in C(\mathcal M(\mathcal A))$. Hence, $U(AB)^\sim \in \mathcal A$ by (\ref{algebra_equality}) and, by $U$-invariance, also $(AB)^\sim \in \mathcal A$. Therefore, Theorem \ref{Theorem_correspondence_theory}(3) yields that $AB \in \mathcal T_{lin}^{2,t}(\mathcal A) \leftrightarrow \mathcal A$, which finishes the proof.
\end{proof}
We say that $\mathcal I \subseteq \mathcal A$ is an $\alpha$-invariant ideal of $\mathcal A$ if it is a closed ideal of $\mathcal A$ such that $f \in \mathcal I$ and $z \in \mathbb C^n$ imply $\alpha_z(f) \in \mathcal I$. We will classify such ideals now.

Recall that the closed ideals $\mathcal I$ in $\mathcal A \cong C(\mathcal M(\mathcal A))$ are in 1-1 correspondence with closed subsets $I$ of $\mathcal M(\mathcal A)$ via
\begin{align*}
\mathcal I_{I} &= \{ f \in \mathcal A;~ x(f) = 0 \text{ for all } x \in I\}\\
&\cong \{ f \in C(\mathcal M(\mathcal A)); ~f(x) = 0 \text{ for all } x \in I\},\\
I_{\mathcal I} &= \{ x \in \mathcal M(\mathcal A); ~x(f) = 0 \text{ for all } f \in \mathcal I\}.
\end{align*}
One easily sees that for $I \subseteq \mathcal M(\mathcal A)$ closed such that $I \cap \mathbb C^n \neq \emptyset$ and $\mathcal I_{I}$ translation-invariant, one necessarily has $I = \mathcal M(\mathcal A)$ or equivalently $\mathcal I_{I} = \{ 0\}$. Therefore, proper translation-invariant ideals ``live at the boundary''.

Since $\mathcal A$ is assumed to be $\alpha$- and $U$-invariant, we obtain induced group actions on $\mathcal M(\mathcal A)$ via
\begin{equation*}
\alpha_z(x)(f) = x(\alpha_z(f)), \quad U(x)(f) = x(Uf).
\end{equation*}
It is not difficult to see that these actions leave $\mathbb C^n$ and $\mathcal M(\mathcal A) \setminus \mathbb C^n$ invariant.
\begin{prop}
A closed subset $I \subset \mathcal M(\mathcal A) \setminus \mathbb C^n$ gives rise to an $\alpha$-invariant ideal $\mathcal I_I$ if and only if $I$ is closed under the action of $\alpha$.
\end{prop}
\begin{proof}
We just demonstrate that $\alpha$-invariance of $I$ implies $\alpha$-invariance of $\mathcal I_I$, the other direction of the proof works similarly. We consider $\mathcal I_I$ as an ideal in $C(\mathcal M(\mathcal A))$. Let $f \in \mathcal I_I$, i.e. $f(x) = 0$ for all $x \in I$. Then, if $z \in \mathbb C^n$ we have $\alpha_z(f)(x) = f(\alpha_z(x)) = 0$ for every $x \in I$, since $\alpha_z(x) \in I$.
\end{proof}
\begin{prop}\label{theorem_ideals}
Let $\mathcal A \subseteq \BUC(\mathbb C^n)$ be a unital, $\alpha$- and $U$-invariant $C^\ast$ subalgebra. If $\mathcal I \subseteq \mathcal A$ is a closed and $\alpha$-invariant ideal of $\mathcal A$, then $\mathcal T_{lin}^{p,t}(\mathcal I)$ is a closed and two-sided $\alpha$-invariant ideal in $\mathcal T_{lin}^{p,t}(\mathcal A)$ for all $t > 0$.
\end{prop}
\begin{proof}
By an argument analogous to Corollary \ref{corollary_ideal}, it suffices to prove that $\mathcal T_{lin}^{2,t}(\mathcal I)$ is an ideal in $\mathcal T_{lin}^{2,t}(\mathcal A)$.

Let $\mathcal I = \mathcal I_{I}$ with $I \subseteq \mathcal M(\mathcal A) \setminus \mathbb C^n$ $\alpha$-invariant. Let $f \in \mathcal I$. Using the invariance of $I$, one obtains $f_{Ux} = 0$ for each $x \in I$. Hence, for each operator $A \in \mathcal T_{lin}^{2,t}(\mathcal I)$ we have $A_{Ux} = 0$ for $x \in I$. Let $B \in \mathcal T_{lin}^{2,t}(\mathcal A)$. Then, we have $(BA)_{Ux} = B_{Ux} A_{Ux}= 0$ and $(AB)_{Ux} = A_{Ux} B_{Ux} = 0$ for each $x \in I$.

As in the proof of Proposition \ref{proposition_unital_algebra} and by the above discussion, we obtain for every $x \in I$
\begin{equation*}
(AB)^\sim(x) = \langle (AB)_{Ux} 1, 1\rangle_{F_t^2} = 0
\end{equation*}
and analogously for $(BA)^\sim$, i.e. both $(AB)^\sim$ and $(BA)^\sim$ extend to functions in $C(\mathcal M(\mathcal A))$ which vanish on $I$, i.e $(AB)^\sim, (BA)^\sim \in \mathcal I$. Thus, Theorem \ref{Theorem_correspondence_theory}(3) yields $AB, BA \in \mathcal T_{lin}^{2,t}(\mathcal I)$, which is therefore an ideal in $\mathcal T_{lin}^{2,t}(\mathcal A)$.
\end{proof}
We are now in the position to drop the assumption that the $C^\ast$ algebra $\mathcal A$ contains the unit element:
\begin{cor}
Let $\mathcal A \subseteq \BUC(\mathbb C^n)$ be an $\alpha$- and $U$-invariant $C^\ast$ subalgebra. Then, we have
\[ \mathcal T_{lin}^{p,t}(\mathcal A) = \mathcal T^{p,t}(\mathcal A) \]
and
\[ \mathcal T_{lin}^{2,t}(\mathcal A) = \mathcal T_{\ast}^{2,t}(\mathcal A) \]
for every $1 < p < \infty$ and $t > 0$. If further $\mathcal I \subseteq \mathcal A$ is an $\alpha$-invariant closed ideal, then $\mathcal T_{lin}^{p,t}(\mathcal I)$ is a two-sided ideal in $\mathcal T_{lin}^{p,t}(\mathcal A)$ for all $t > 0$, $1 < p < \infty$.
\end{cor}
\begin{proof}
Assume $\mathcal A$ is not unital. $\mathcal A \oplus \mathbb C 1 \subset \BUC(\mathbb C^n)$ is an $\alpha$- and $U$-invariant unital $C^\ast$ subalgebra of $\BUC(\mathbb C^n)$ in which $\mathcal A$ is an $\alpha$-invariant ideal. Therefore, $\mathcal T_{lin}^{2,t}(\mathcal A)$ is an ideal in $\mathcal T_{lin}^{2,t}(\mathcal A \oplus \mathbb C 1)$ and in particular a $C^\ast$ algebra and needs to agree with $\mathcal T_{\ast}^{2,t}(\mathcal A)$. Further, $\mathcal I$ is also an ideal in $\mathcal A \oplus \mathbb C1$, therefore $\mathcal T_{lin}^{2,t}(\mathcal I)$ is an ideal in $\mathcal T_{lin}^{2,t}(\mathcal A \oplus \mathbb C1)$ and also in $\mathcal T_{lin}^{2,t}(\mathcal A)$.
\end{proof}

Combining all results, we have obtain a proof of Theorem \ref{main_theorem}.

\section{Discussion}
First, let us provide a class of examples which give proper linearly generated subalgebras of $\mathcal T^{p,t}$ according to Theorem \ref{main_theorem}:

Let $\mathcal O \subset \mathbb C^n$ be a closed, non-empty subset. We set
\[ \mathcal A_{\mathcal O} := \{ f \in \BUC(\mathbb C^n); ~ \alpha_w(f) = f \text{ for all } w \in \mathcal O\}. \]
We might assume without loss of generality that $\mathcal O$ is an additive subgroup of $\mathbb C^n$. 
One easily sees that $\mathcal A_{\mathcal O}$ is a unital $\alpha$- and $U$-invariant $C^\ast$ algebra. Hence, the results from Theorem \ref{main_theorem} apply:
\[ \mathcal T_{lin}^{p,t}(\mathcal A_{\mathcal O}) = \mathcal T^{p,t}(\mathcal A_{\mathcal O}). \]
It is now not hard to prove that
\begin{align*}
\mathcal T_{lin}^{p,t}&(\mathcal A_{\mathcal O}) \\
&= \mathcal T^{p,t}(\{ f \in L^\infty(\mathbb C^n); ~ \alpha_w(f) = f \text{ for all } w \in \mathcal O\})\\
&= \{ A \in \mathcal T^{p,t}; ~ \alpha_w(A) = A \text{ for all } w \in \mathcal O\}.
\end{align*}
\begin{rem}
If we let $\mathcal O = \mathcal L$ be a Lagrangian subspace of $\mathbb C^n$ (with the special case $\mathcal L = \{ 0 \} \times i \mathbb R^n \subset \mathbb C^n$) we obtain the class of Lagrangian-invariant (respectively horizontal) Toeplitz operators. These have been studied in \cite{Esmeral_Vasilevski2016}. Among other results, it has been proven there (using different methods) that
\begin{align*}
\mathcal T_{lin}^{2,1}&(\{f  \in L^\infty(\mathbb C^n); ~ \alpha_w(f) = f, ~ w \in \mathcal L\})\\
 &= \mathcal T_{\ast}^{2,1}(\{f  \in L^\infty(\mathbb C^n); ~ \alpha_w(f) = f, ~ w \in \mathcal L\}),
\end{align*}
which turns out to be a special case of the above proposition.
\end{rem}
Let $\mathcal O \subset \mathbb C^n$ be as above and fix $z \in (\mathcal O)^\perp \setminus \{ 0\}$. Let us denote by
\[ \mathcal I_{\mathcal O, z} := \{ f \in \mathcal A_{\mathcal O}; ~ \alpha_w f(\lambda z) \to 0,~ \lambda \in \mathbb R,~ |\lambda| \to \infty \text{ for all } w \in \operatorname{span} \mathcal O \} \]
the set of all functions in $\mathcal A_{\mathcal O}$ which vanish in the direction $z$ orthogonal to $\mathcal O$ at infinity. $\mathcal I_{\mathcal O, z}$ is an $\alpha$-invariant ideal of $\mathcal A_{\mathcal O}$ to which the results from Theorem \ref{main_theorem} apply, i.e. $\mathcal T_{lin}^{p,t}(\mathcal I_{\mathcal O, z})$ is a two-sided ideal in $\mathcal T_{lin}^{p,t}(\mathcal A_{\mathcal O})$ for all $t > 0$.
\newline \ \\
We end with the following comments:

\begin{enumerate}[1)]
\item In \cite{Xia}, J. Xia also proved equality of the Toeplitz algebra $\mathcal T^{2,1}$ with the Banach algebra generated by weakly localized operators. Based on our Theorem \ref{Toeplitz_algebra_linearly_generated}, R. Hagger recently proved the same equality for every $t > 0$ and $1 < p < \infty$ \cite{Hagger2020}.
\item In this work, we restricted ourselves to the case $1 < p < \infty$. Nevertheless, one can try to establish the Correspondence Theory also over the Fock spaces $F_t^1, f_t^\infty$ and $F_t^\infty$ (cf. \cite{Zhu} for definitions and properties). If one naively tries to imitate the approach used here over $f_t^\infty$, one encounters a problem when defining the convolution between $f \in L^\infty(\mathbb C^n)$ and $A \in \mathcal N(f_t^\infty)$: Since $f_t^\infty$ is not reflexive, $A \ast f$ is naturally contained in $\mathcal L((f_t^\infty)')$ and not in $\mathcal L(f_t^\infty)$! This problem can be fixed, since the Correspondence Theory only needs $f \ast \mathcal R_t = T_f^t \in \mathcal L(f_t^\infty)$, which is indeed true. With some additional efforts, a version of the Correspondence Theory and its applications can then be established over $f_t^\infty$ and, using the dualities $(f_t^\infty)' \cong F_t^1$ and $(F_t^1)' \cong F_t^\infty$, also over those Fock spaces. Details of this will appear elsewhere.
\item Theorem \ref{main_theorem} needs both the $U$-invariance and the self-adjointness of $\mathcal D_0$. As we have seen, this is only important for one direction of the proof, while the other direction remains true without these assumptions. It is natural to ask whether these assumptions are in fact necessary. As for now, this is an open problem to the author. 
\item One might also wonder if there are $\alpha$-invariant and closed $\mathcal D_0 \subseteq \buc(\mathbb C^n)$ such that $\mathcal T_{lin}^{p,t}(\mathcal D_0) = \mathcal T^{p,t}(\mathcal D_0)$ holds just for some $t > 0$, but not all such $t$. No such example is known to the author.
\end{enumerate}

\section{Acknowledgements}
The author is grateful for the support of his supervisor, Prof. Wolfram Bauer, during the preparation of this work. He also appreciates the anonymous referee's helpful comments on this paper.

\bibliographystyle{amsplain}
\bibliography{References}

\bigskip

\noindent
Robert Fulsche\\
\href{fulsche@math.uni-hannover.de}{\Letter fulsche@math.uni-hannover.de}
\\

\noindent
Institut f\"{u}r Analysis\\
Leibniz Universit\"at Hannover\\
Welfengarten 1\\
30167 Hannover\\
GERMANY

\end{document}